\RequirePackage{snapshot}
\documentclass[11pt,svgnames]{article}

\usepackage{latex_styles_common}
\usepackage[text={6in,8.5in},centering]{geometry}
\usepackage[font=small]{caption}
\usepackage[sort]{cite}
\usepackage[ruled,vlined,linesnumbered]{algorithm2e} 
\usepackage[noabbrev,nameinlink,capitalize]{cleveref}
\usepackage{booktabs}
\usepackage{pgfplotstable}

\usepackage{enumitem}
\newlist{propenum}{enumerate}{1} %
\setlist[propenum]{label=\alph*{\rm)}, ref=\theproposition(\alph*)}
\crefalias{propenumi}{proposition}
\newlist{corenum}{enumerate}{1} %
\setlist[corenum]{label=\alph*{\rm)}, ref=\thecorollary(\alph*)}
\crefalias{corenumi}{corollary}
\crefname{assumption}{Assumption}{Assumptions}
\usepackage{amsthm,thmtools}
\declaretheorem[name=Theorem,numberwithin=section]{theorem}
\declaretheorem[name=Definition,style=definition,numberlike=theorem]{definition}
\declaretheorem[name=Example,style=definition,numberlike=theorem,qed=\qedsymbol]{example}
\declaretheorem[name=Proposition,numberlike=theorem]{proposition}
\declaretheorem[name=Corollary,numberlike=theorem]{corollary}
\declaretheorem[name=Assumption,numberlike=theorem]{assumption}

\numberwithin{equation}{section}
\numberwithin{theorem}{section}
\numberwithin{figure}{section}
\numberwithin{table}{section}
\numberwithin{algocf}{section}

\input{macros}
\newcommand{\polar}{^\circ}
\newcommand{\convA}{\widehat\Ascr}
\newcommand{\Face}[1]{\Fscr_{\scriptscriptstyle #1}}
\newcommand{\FaceA}{\Fscr_{\!\!\scriptscriptstyle\Ascr}}

\newcommand{\FaceCp}{\Fscr_{\!\!\scriptscriptstyle\Cscr\polar}}
\newcommand{\gauge}{\gamma}
\newcommand{\As}{_{\!\scriptscriptstyle\Ascr}}
\newcommand{\Asi}{_{\!\scriptscriptstyle\Ascr_i}}
\newcommand{\Aso}{_{\!\scriptscriptstyle\Ascr_1}}
\newcommand{\Ast}{_{\!\scriptscriptstyle\Ascr_2}}

\newcommand{\Cs}{_{\scriptscriptstyle\Cscr}}
\newcommand{\Csp}{_{\scriptscriptstyle\Cscr^\circ}}

\newcommand{\maxconv}{\!\mathop{\diamond}}

\DeclareMathOperator{\suppa}{\Sscr\As}
\DeclareMathOperator{\supp}{\hbox{\rm\textbf{supp}}}
\newcommand{\subAsum}{_{\scriptscriptstyle\Ascr_1+\Ascr_2}}

\begin{document}

\pdfinfo{/Author (Michael P. Friedlander)
         /Title (Insert Title Here)
         /Keywords (keyword1, keyword2, keyword3)}

       \title{Polar Alignment and Atomic Decomposition%
         \footnote{Department of Computer Science, University of British
           Columbia, 2366 Main Mall, Vancouver, BC, V6R 1Z4, Canada. Email:
           \texttt{zhenanf@cs.ubc.ca}, \texttt{clatar1@gmail.com},
           \texttt{yifan.0.sun@gmail.com}, \texttt{michael.friedlander@ubc.ca}.
           Research supported by ONR award N00014-17-1-2009. }%
       } \author{Zhenan Fan, Halyun Jeong, Yifan Sun, Michael P. Friedlander}
       \date{December 11, 2019}

\maketitle

\begin{abstract}
  Structured optimization uses a prescribed set of atoms to assemble a solution
  that fits a model to data. Polarity, which extends the familiar notion of
  orthogonality from linear sets to general convex sets, plays a special role in
  a simple and geometric form of convex duality. This duality correspondence
  yields a general notion of alignment that leads to an intuitive and complete
  description of how atoms participate in the final decomposition of the
  solution. The resulting geometric perspective leads to variations of existing
  algorithms effective for large-scale problems. We illustrate these ideas with
  many examples, including applications in matrix completion and morphological
  component analysis for the separation of mixtures of signals.
\end{abstract}

\makeatletter
\renewcommand\tableofcontents{%
  \@starttoc{toc}%
}
\makeatother
\begin{quote}
  \centering
  \textbf{Contents}
  \\
  \setcounter{tocdepth}{1}
  \tableofcontents
\end{quote}
\clearpage

\section{Introduction}

Convex optimization provides a valuable computational framework that renders
many problems tractable because of the range of powerful algorithms that can be
brought to the task. The key is that a certain mathematical structure---i.e.,
convexity of the functions and sets defining the problem---lays open an enormous
range of theoretical and algorithmic tools that lend themeselves astonishingly
well to computation. There are limits, however, to the scalability of
general-purpose algorithms for convex optimization. As has been recognized in
the optimization and related communities for at least the past decade,
significant efficiencies can be gained by acknowledging the latent structure in
the solution itself, coupled with the overarching structure provided by
convexity.

Structured optimization proceeds along these lines by using a prescribed set of
atoms from which to assemble an optimal solution. In effect, the atoms selected
to participate in forming the solution decompose the model into simpler parts,
which offers opportunities for algorithmic efficiency in solving the
optimization problem. From a modeling point of view, the particular atoms that
constitute the computed solution often represent key explanatory components of a
model. An atomic decomposition thus provides the principal components of a
solution, i.e., its most informative features. %

Our purpose with this paper is to describe the rich convex geometry that
underlies atomic decomposition. The path we follow builds on the duality
inherent in convex cones: every convex cone is paired with a polar cone. The
extreme rays of any one of these cones is in some sense \emph{aligned} with
certain extreme rays of its polar cone. Brought into the context of atomic
decomposition, this notion of polar alignment provides a theoretical framework
for identifying the atoms that participate in a decomposition. This approach
facilitates certain algorithmic design patterns that promote computational
efficiency, as we demonstrate with concrete examples.

\section{Atomic decomposition}

The decomposition of a fixed vector $x\in\Re^n$ with respect to a set
of atoms $\Ascr\subset\Re^n$ is given by the sum
\begin{equation}
  \label{eq:1}
  x = \sum_{a\in\Ascr} c_a a, \quad c_a\ge0 \quad \forall a\in\Ascr.
\end{equation}
Each coefficient $c_a$ measures the contribution of the corresponding atom $a$
toward the construction of $x$. %
We are particularly interested in the question of determining which atoms are
essential to expressing $x$ as a positive superposition. Let
\begin{equation}
  \label{eq:2}
  \gauge\As(x) = \inf_{c_a}\Set{ \sum_{a\in\Ascr}c_a | x = \sum_{a\in\Ascr}c_aa,\ c_a \ge0\ \forall a\in\Ascr }
\end{equation}
be the minimal sum of weights over all valid atomic decompositions. The
significant atoms (those that \emph{support} the vector $x$) are those that
contribute positively in forming the minimal sum. We are thus led to the
following definition.

\begin{definition}[Support set] \label{def:support} A set
  $\Sscr_{\Ascr}(x)\subset\Ascr$ is a \emph{support set} for $x$ with
  respect to $\Ascr$ if every element $a\in\Sscr_\Ascr(x)$ has a
  coefficient $c_a$ from~\eqref{eq:1} that is strictly positive. That
  is,
  \begin{equation}
    \label{eq:3}
    \gauge\As(x) = \sum_{\mathclap{a\in\suppa(x)}} c_a,
    \qquad x = \sum_{\mathclap{a\in\suppa(x)}} c_aa,
    \text{and} c_a > 0\enspace \forall a\in\suppa(x).
  \end{equation}
  The set $\supp\As(x)$ is defined as the set of all support
  sets. Thus, any $\Sscr\in\supp\As(x)$ is a valid support set. \qed
\end{definition}

For any given atomic decomposition of the solution to an optimization problem,
the atoms $a\in\Ascr$ with large coefficients $c_a$ correspond to atoms that are
most significant in the minimization process. In the simplest case, the atoms
$\Ascr$ may be taken as the collection of canonical unit vectors $\{\pm
e_1,\ldots,\pm e_n\}$, and then the significant atoms correspond to the most
significant variables $x_j$ in the vector $x=(x_1,\ldots,x_n)$. Under
\cref{def:support}, this interpretation extends to arbitrary atomic sets.

This generic model for atomic decompositions was promoted by Chen
et al.~\cite{cds98,chendonosaun:2001} in the context of sparse signal
decomposition, and more recently by Chandrasekaran et
al.~\cite{chandrasekaran2012convex}, who are concerned with obtaining sparse
solutions to linear inverse problems. In the general framework outlined by
Chandrasekaran et al., the gauge function $\gauge\As$ can be used to define a
general convex optimization problem suitable for recovering a ground-truth
solution from a relatively small number of observations.

\subsection{Approach}
\label{sec:approach}

The convex function $\gamma\As$ is equivalent to the Minkowski
functional~\cite[Section~15]{rockafellar1970convex} to the convex hull
$\convA\:=\conv(\Ascr\cup\{0\})$; see
\cref{prop-guage-equivalence}. As we describe in
\cref{sec-convexanalysis}, the Minkowski and support functions
\begin{equation}\label{eq:4}
  \gauge\As(x) = \inf\set{\lambda\ge0|x\in\lambda\convA}
  \textt{and}
  \sigma\As(x) = \sup\set{\ip{x}{z} | z\in\convA}
\end{equation}
to the set $\Ascr$ form a dual pairing under a polarity operation. One of the
defining properties of this dual pairing is that it satisfies the \emph{polar
  inequality}\cite[Section~15]{rockafellar1970convex}
\begin{equation}
  \label{eq:polar-inequality}
  \ip x z \le \gauge\As(x)\cdot\sigma\As(z) \quad \forall (x,z)\in\dom\gauge\As\times\dom\sigma\As.
\end{equation}

\begin{definition}[Alignment] \label{def:alignment} A pair
  $(x,z)\in\Real^n\times\Real^n$ is \emph{aligned} with respect to the atomic
  set $\Ascr$, i.e., $x$ and $z$ are $\Ascr$-aligned, if the polar
  inequality~\eqref{eq:polar-inequality} holds as an equation.
\end{definition}

\begin{figure}[t]
  \centering
  \includegraphics[page=4]{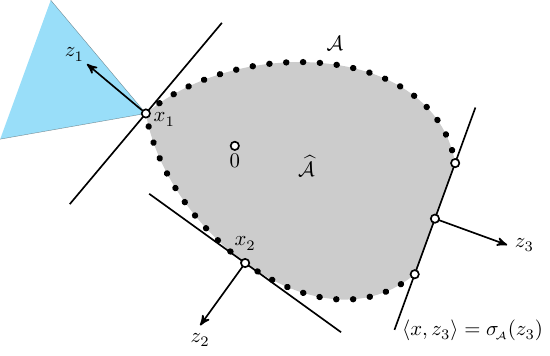}
  \caption{The set of atoms in the set $\Ascr$ generally (but not necessarily)
    defines the boundary of the convex hull $\convA$. The essential
    atoms $\Escr\As(z)$ exposed by a vector $z$ lie on the supporting
    hyperplane $\set{x|\ip x z=\sigma\As(z)}$.}
  \label{fig:essential-atoms}
\end{figure}

This general notion of alignment follows from the special case where
$\Ascr=\set{x|\norm{x}_2\le1}$ is the unit 2-norm ball, which results in
$\gauge\As=\sigma\As=\norm{\cdot}_2$. In that case, the polar
inequality~\eqref{eq:polar-inequality} then reduces to the well-known
Cauchy-Schwartz inequality
\begin{equation*}
  \label{eq:cauchy-inequality}
  \ip x z \le \norm{x}_2\cdot\norm{z}_2,
\end{equation*}
which holds as an equation if and only if $x$ and $z$ are aligned in
the usual sense: there exists a nonnegative scalar $\alpha$ such that
$x=\alpha z$. Our notion of alignment captures other important special
cases, including the H\"older inequality, which is a
special case of~\eqref{eq:polar-inequality} in which $\Ascr$ is the unit
$p$-norm ball, with $p\in[1,\infty]$.

A rich convex geometry 
underlies this general notion of alignment, and plays a role in 
identifying the atoms important for the decomposition~\eqref{eq:2}.
Suppose that a vector $z$ is $\Ascr$-aligned with $x$. As we will demonstrate,
all atoms $a\in\Ascr$ that participate significantly in a decomposition of $x$
must be contained in the set of exposed atoms, i.e.,
\begin{equation} \label{eq:13}
  \suppa(x)\subseteq\Escr\As(z):=\Set{a\in\Ascr\cup\{0\} | \ip a z = \sigma\As(z) }.
\end{equation}
Note that the convex hull of the exposed set $\Escr\As(z)$ forms a face of
$\convA$ exposed by the vector $z$. Because all of the atoms $a\in\Escr\As(z)$
necessarily have unit gauge value, i.e., $\gauge\As(a)=1$, the condition $\ip a
z = \sigma\As(z)$ then implies that significant atoms must also be
$\Ascr$-aligned with $z$. \Cref{fig:essential-atoms} presents a visualization of
this concept.

\subsection{Examples}

There are many varieties of atomic sets and recognizable convex
regularizers used to obtain sparse decompositions.
Chandrasekaran et al.~\cite{chandrasekaran2012convex} and
Jaggi~\cite{jaggi2013revisiting} both give extensive lists of atoms
and the norms that they induce, as well as their applications in practice.
Here we provide several simple examples that illustrate the variety of
ways in which vectors can be aligned.

\begin{example}[One norm] \label{example-one-norm} Let
  $\Ascr = \{\pm e_1,\ldots,\pm e_n\}$ be the signed standard basis
  vectors. This atomic set induces the 1-norm, which is the canonical
  example of a sparsifying convex penalty, and is paired with its dual
  $\infty$-norm:
  \[
    \gauge\As(x) = \|x\|_1 \text{and} \sigma\As(z) = \|z\|_\infty.
  \]
  The polar inequality~\eqref{eq:polar-inequality} reduces to
  H\"older's inequality for these norms---i.e.,
  $\ip x z \le \|x\|_1\cdot\|z\|_\infty$. As is well known, this holds
  with equality---i.e., $x$ and $z$ are $\Ascr$-aligned---if and only
  if
\begin{equation*}
  x_i \neq 0
  \quad\Longrightarrow\quad
  \sign(x_i) z_i = \max_j\,|z_j| \quad \forall i=1,\ldots,n.
\end{equation*}
Thus, alignment of the pair $(x,z)$ with respect to the atomic set
$\Ascr$ is equivalent to the statement that $\suppa(x)\subset\Escr_\Ascr(z)$,
with
\[
  \suppa(x) = \set{\sign(x_i)e_i | x_i \neq 0}
  \text{and}
  \Escr_\Ascr(z) = \set{\sign(z_i)e_i | |z_i| =  \max_j\,|z_j|}.
\]

This condition also characterizes an optimality condition. For example, consider
the LASSO~\cite{tib96} problem
\begin{equation*}
  \minimize{x}\enspace \half\|Ax-b\|_2^2 \enspace\st\enspace\|x\|_1\le\tau,
\end{equation*}
where $\tau$ is a non-negative parameter. It is straightforward to verify that
$x$ is optimal if and only if $\suppa(x)\subset\Escr_\Ascr(z)$ where
$z=A\T(b-Ax)$ is the negative gradient of the objective.
\cref{sec-manifestations} describes in more detail the connection between
optimality and alignment.
\end{example}

\begin{example}[Nuclear norm] \label{example-nuclear-norm}

  The nuclear norm, or Schatten 1-norm, of a matrix is the spectral analog to
  the vector 1-norm. The nuclear norm and its dual spectral norm can be obtained
  via the atomic set $\Ascr = \set{uv^T | \|u\|_2=\|v\|_2 = 1}$ of normalized
  $n$-by-$m$ rank-1 matrices. Then for matrices $X$ and $Z$,
  \[
    \gauge\As(X) = \|X\|_* \text{and} \sigma\As(Z) = \sigma\submax(Z).
  \]
  These are, respectively, the nuclear and spectral norms of $X$ and $Z$---i.e.,
  the sum of singular values of $X$ and the maximum singular value of $Z$. The
  atomic description of these functions is consistent with the notion that the
  nuclear norm is a convex function that promotes low rank (e.g., sparsity with
  respect to rank-1 matrices) \cite{recht2010guaranteed}. Define the trace inner
  product $\ip X Z := \trace X\T Z$. The alignment condition $\ip X Z =
  \norm{X}_1\cdot \norm{Z}_\infty$ holds when $X$ and $Z$ have a simultaneously
  ordered singular value decomposition (SVD). In particular, if $X$ is rank $r$,
  then
  \[
    X = \sum_{i=1}^r c^{}_i u^{}_iv_i^T
    \text{and}
    Z = \sum_{i=1}^{\mathclap{\min\{m,n\}}} s^{}_iu^{}_iv_i^T
  \]
  are the SVDs of $X$ and $Z$, where the singular values are ordered as
  \[
    c_1 \geq \cdots \geq c_r > 0, \textt{and} s_1 = \cdots = s_d >
    s_{d+1} \geq \cdots \geq s_{\min\{m,n\}} \geq 0.
  \]
  By this description,  
  \[
    \suppa(X)=\set{u^{}_1v_1^T,\ldots,u^{}_rv_r^T} \text{and}
    \Escr\As(Z)=\set{u^{}_1v_1^T,\ldots,u^{}_dv_d^T}.
  \]
  The inclusion~\eqref{eq:13}, which identifies the support as a subset of the
  exposed atoms, implies $d \geq r$. Thus, the singular vectors of $Z$
  corresponding to the $d$ singular values $s_1,\ldots,s_d$ contain the singular
  values of $X$. Note that this can also be proven as a consequence of von
  Neumann's trace inequality \cite{vonneumann:1937,lewis:1995}. This property is
  used by Friedlander et al. \cite{friedlander2016low} for the construction of
  dual methods for low-rank semidefinite optimization.
\end{example}

\begin{example}[Linear subspaces] \label{example-linear-subspaces}

  Suppose that the set of atoms $\Ascr$ contains all the elements of a
  linear subspace $\Lscr$. In this case, the gauge $\gamma_\Lscr(x)$
  is finite only if $x$ is in $\Lscr$, and similarly, the support
  function $\sigma_\Lscr(z)$ is finite only if $z$ is in its
  orthogonal complement $\Lscr^\perp$. In particular, because $\Lscr$ and
  $\Lscr^\perp$ are cones,
  \[
    \gauge_\Lscr(x) = \delta_\Lscr(x)
    \text{and}
    \sigma_\Lscr(z) = \delta_{\Lscr^\perp}(z),
  \]
  where $\delta\Cs(v)$ is the indicator to a set $\Cscr$, which evaluates to 0
  if $v\in\Cscr$ and to $+\infty$ otherwise. The respective domains of the
  gauge and support functions are thus $\Lscr$ and $\Lscr^\perp$. It
  follows that, under the atomic set $\Lscr$, the vectors $x$ and $z$ are
  $\Lscr$-aligned if and only if $x\in\Lscr$ and $z\in\Lscr^\perp$. Thus, the aligned
  vectors are orthogonal.
\end{example}

\subsection{Applications and prior work}

One of the main implications of our approach is its usefulness in
using dual methods for discovering atomic decompositions. A dual
optimization method can be interpreted as solving for an aligning
vector $z$ that exposes the support of a primal solution $x$. If the
number of exposed atoms is small, a solution $x$ of the primal problem
can be resolved over the reduced support, but without the atomic
regularization, which may be computationally much cheaper~\cite{2015arXiv151102204F} or better
conditioned~\cite{negahban2012restricted}. Alternatively, two-metric
methods can be designed to act differently on a primal iterate's
suspected support~\cite{gafni1984two}. In many applications, such as
feature selection, knowing the support itself may be sufficient.  The
conditions under which such a $z$ can often be found occurs in several
applications, as we describe with various examples throughout the paper.

\paragraph{Machine learning}

The regularized optimization problem described in \cref{sec-manifestations}
frequently appear in applications of machine learning for the purpose of model
complexity reduction. The most popular use cases are the vector 1-norm
$\gauge\As(x) = \|x\|_1$ in feature selection~\cite{tibshirani1996regression},
its group-norm variant~\cite{jacob2009group}, and the nuclear norm
$\gauge\As(X) = \|X\|_*$ in matrix completion \cite{recht2010guaranteed}.
However, many other sparsity-promoting regularizers appear in practice
\cite{zeng2014ordered}. Although the unconstrained formulation is most popular,
particularly when the proximal operator is computationally
convenient~\cite{parikh2013proximal}, the gauge-constrained formulation is
frequently used and solved via the conditional gradient method
\cite{frank1956algorithm,dunn1978conditional,jaggi2013revisiting}. Popular dual
methods, which iterate over a dual variable $z^{(k)}\equiv -\nabla f(x^{(k)})$
but maintain the corresponding primal variable $x^{(k)}$ only implicitly, include bundle methods
\cite{lemarechal1981bundle} and dual averaging
\cite{xiao2010dual,duchi2012dual}.

\paragraph{Linear conic optimization}

Conic programs are a cornerstone of convex optimization. The nonnegative cone
$\Re^n_+$, the second-order cone $\Qscr_+^{n+1} = \set{(x,\tau) | \|x\|_2 \leq
  \tau}$, and the semidefinite cone $\Sscr_+^n = \set{X | u\T Xu\ge0\ \forall
  u}$, respectively, give rise to linear, second-order, and semidefinite
programs. These problem classes capture an enormous range of important models,
and can be solved efficiently by a variety of algorithms, including interior
methods \cite{karmarkar1984new,nen94,rene:2001}. Conic programs and their
associated solvers are key ingredients for general purpose optimization software
packages such as YALMIP \cite{lofberg2004yalmip} and CVX \cite{grant2008cvx}.
The alignment conditions for these specific cones have been exploited in dual
methods, such as in the spectral bundle method for large-scale semidefinite
programming \cite{helmberg2000spectral}. \Cref{example-conic-opt} demonstrates
this alignment principle in the context of conic optimization.

\paragraph{Gauge optimization} 

The class of gauge optimization problems, as defined by Freund's 1987 seminal
work~\cite{freund1987dual}, can be simply stated: find the element of a convex
set that is minimal with respect to a gauge. These conceptually simple problems
appear in a remarkable array of applications, and include parts of sparse
optimization and all of conic
optimization~\cite[Example~1.3]{friedlander2014gauge}. This class of
optimization problems admits a duality relationship different from classical
Lagrange duality, and is founded on the polar inequality.
In this context, the polar inequality provides an analogue to weak duality,
well-known in Lagrange duality, which guarantees that any
feasible primal value provides an upper bound for any feasible dual
value. In the gauge optimization context, a primal-dual pair $(x,z)$ is optimal
if and only if the polar inequality holds as an equation, which under
\cref{def:alignment} implies that $x$ and $z$ are
aligned. %
The connection between polar alignment and optimality is
discussed further in \cref{sec-gauge-optimization}.

\paragraph{Two-stage methods} In sparse optimization, two-stage methods first
identify the primal variable support, and then solve the problem over a reduced
support \cite{ko1994iterative,cristofari2017two}. If the support is sparse
enough, the second problem may be computationally much cheaper, either because
it allows for faster Newton-like methods, or because of better conditioning
\cite{negahban2012restricted}. The atomic alignment principles we describe in
\cref{sec-convexanalysis-atomic} give a general recipe for extracting primal
variable support from a computed dual variable, which at optimality is aligned
with the primal variable; see \cref{sec-manifestations}. This property forms the
basis for our approach to morphological component analysis, described in
\cref{sec-morphological-component-analysis}.

\section{Alignment with respect to general convex sets} \label{sec-convexanalysis}

The alignment principles we develop depend on basic notions of convex sets and
their supporting hyperplanes. Gauges and support functions, defined
in~\eqref{eq:4}, facilitate many of the needed derivations. Define the conic
extension of any set $\Dscr\subset\Real^n$ by
\[
  \cone\Dscr = \set{\alpha d | d\in\Dscr,\ \alpha\ge0}.
\]
Throughout the paper, we use the symbol $\Cscr$
to denote a general convex set in $\Re^n$. %
The following blanket assumption, which holds throughout the paper, ensures a
desirable symmetry between a set and its polar, as explained
in~\cref{sec-polarity}. This assumption considerably simplifies our analysis and
fortunately holds for many of the most important and relevant examples.
\begin{assumption}[Origin containment] \label{blanket-assumption}
  The set $\Cscr\subset\Re^n$ is closed convex and contains the origin.
\end{assumption}

\subsection{Polarity} \label{sec-polarity}

Our notion of alignment is based on the polarity of convex
sets. Polarity is most intuitive in the context of convex cones, which
are convex sets closed under positive scaling, i.e., the
set $\Kscr$ is a convex cone if $\alpha\Kscr\subset\Kscr$ for all
$\alpha>0$ and $\Kscr + \Kscr \subset \Kscr$. Its polar
\begin{equation} \label{eq-polar-cone}
  \Kscr\polar=\set{z | \ip x z \le 0 \ \forall x\in\Kscr}
\end{equation}
is also a convex cone, and its vectors make an oblique angle (i.e., a
nonpositive inner product) with every vector in $\Kscr$. For a general
convex set $\Cscr$, its polar is defined as the convex set
\begin{equation} \label{eq-polar-set}
  \Cscr\polar=\set{z | \ip x z \le 1\ \forall x\in\Cscr }.
\end{equation}

One way to connect the polarity definitions~\eqref{eq-polar-cone}
and~\eqref{eq-polar-set} is by ``lifting'' the set $\Cscr$ and its polar
$\Cscr\polar$ and embedding them into opposing cones in $\Re^{n+1}$:
\[
  \Kscr\Cs \defd \cone(\Cscr\times\{1\})
  \text{and}
  \Kscr\Cs\polar \defd \cone(\Cscr\polar\times\{-1\}).
\]
Then for any nonzero $(n+1)$-vectors $\xbar\in\Kscr$ and $\zbar\in\Kscr\polar$,
there exist positive scalars $\alpha_x$ and $\alpha_z$, and vectors $x\in\Cscr$ and $z\in\Cscr\polar$, such that
\begin{equation} \label{eq:10}
  \ip\xbar\zbar = \ip*{\alpha_x\pmat{x\\1}}{\,\alpha_z\pmat{\phantom-z\\-1}}
  = \alpha_x\cdot\alpha_z(\ip x z - 1)\le0,
\end{equation}
where the last inequality follows from the polar definition in
\eqref{eq-polar-set}. This last inequality confirms that the cones $\Kscr\Cs$
and $\Kscr\Cs\polar$ are polar to each other under definition~\eqref{eq-polar-cone}.

The blanket assumption that $\Cscr$ is closed and contains the origin
(\cref{blanket-assumption}) yields a special symmetry because then
the polar $\Cscr\polar$ also contains the origin and
$\Cscr^{\circ\circ}=\Cscr$
\cite[Theorem~14.5]{rockafellar1970convex}. This is one of the reasons
why we define $\convA=\conv(\Ascr\cup\{0\})$ to include the origin.%

The polar pair $\Cscr$ and $\Cscr\polar$ can be said to generate the
corresponding gauge and support functions $\gauge\Cs$ and $\sigma\Cs$, as we
show below. It follows immediately from~\eqref{eq:4} that the gauge and support
functions are positively homogeneous, i.e., $\gauge\Cs(\alpha x) =
\alpha\gauge\Cs(x)$ for all $\alpha\ge0$, and similarly for $\sigma\Cs$. Thus
the epigraphs for these functions are convex cones. Moreover, the unit level sets for
these functions are the sets that define them:
\begin{equation} \label{eq:15}
  \Cscr=\set{x|\gauge\Cs(x)\le1}
  \text{and}
  \Cscr\polar = \set{z|\sigma\Cs(z)\le1}.
\end{equation}
It thus follows that
\begin{equation}\label{eq:26} 
  \epi\gauge\Cs = \cone(\Cscr\times\{1\})
  \text{and}
  \epi\sigma\Cs = \cone(\Cscr\polar\times\{1\}).
\end{equation}
\cref{fig:gauge-epi} shows a visualization of the epigraph of the gauge to $\Cscr$.

\begin{figure}[t]
  \centering
  \includegraphics[page=7]{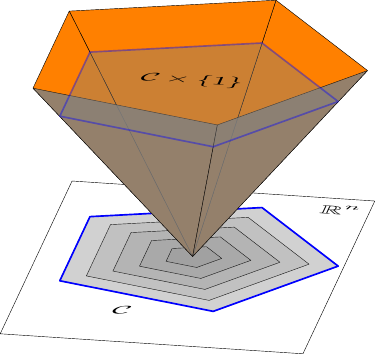}
  \caption{The epigraph of the gauge $\gauge\Cs$ is the cone in
    $\Re^n\times\Real$ generated by the set $\Cscr\subset\Re^n$; see~\eqref{eq:26}.}
  \label{fig:gauge-epi}
\end{figure}

For a set $\Cscr$, the recession cone contains the set of unbounded
directions:
\begin{equation} \label{eq-recession-cone}
  \rec \Cscr :=
  \set{d | \mbox{$x+\lambda d \in \Cscr$ for every $\lambda\ge0$ and $x\in\Cscr$}}.
\end{equation}
See~\cref{fig-recession-example} for an illustration.
Vectors in the recession cone can also be thought of as ``horizon
points'' of $\Cscr$ \cite[p.~60]{rockafellar1970convex}. With respect
to the gauge and support functions to the set $\Cscr$, vectors
$u\in\rec\Cscr$ have the property that $\gauge\Cs(u) = 0$ and
$\sigma\Cs(u) = +\infty$; see \cref{prop-support-properties}. We must
therefore be prepared to consider cases where these functions can take
on infinite values. Far from being a nuisance, this property is useful
in modelling important cases in optimization.

The following proposition collects standard results regarding gauge
and support functions and establishes the polarity correspondence
between these two functions.  The proofs of these claims can be found in
standard texts, notably Rockafellar~\cite{rockafellar1970convex} and
Hiriart-Urruty and Lemarechal~\cite{hiriart-urruty01}. These proofs
typically rely on properties of conjugate functions. Because our
overall theoretical development does not require conjugacy, we provide
self-contained proofs that depend only on properties of closed convex
sets.
  
\begin{proposition}[Properties of gauges and support
  functions] \label{prop-support-properties} Let $\Cscr\subset\Real^n$ be a
  closed convex set that contains the origin, and $\Dscr\subset\Real^n$ be an
  arbitrary set. The following statements hold.
  \begin{propenum}
    \item (Closure and convex hull) $\sigma_{\scriptscriptstyle\Dscr} = \sigma_{\cl\conv\Dscr}$.
    \item (Polarity) \label{prop-support-properties-polarity}
      $\gauge\Cs = \sigma\Csp$.
    \item (Linear transformation) \label{prop-linear-transform} For a
      linear operator $M$ with adjoint $M^*$,
      \[
        \gauge_{\scriptscriptstyle M\inv\Cscr}(x) = \gauge_{\Cscr}(Mx),
          \text{and}
        \sigma_{\scriptscriptstyle M\Cscr}(z) = \sigma\Cs(M^* z),
      \]
      where we interpret $M\inv\Cscr = \set{x \mid Mx\in\Cscr}$ and $M\Cscr=\set{M x \mid x\in\Cscr}$.
    \item (Scaling) \label{prop-scaling}
      $\alpha\gauge\Cs=\gauge_{\frac1\alpha\Cscr}$ and
      $\alpha\sigma\Cs=\sigma_{\alpha\Cscr}$ for all $\alpha>0$.
    \item (Domains) \label{prop-support-properties-domain}
      $\dom\gauge\Cs = \cone\Cscr$ \ and \
      $\dom\sigma\Cs=(\rec\Cscr)\polar$.
    \item (Bijection) 
    $\Cscr = \set{x\in\Re^n| \ip{x}{z} \leq \sigma\Cs(z) \mbox{\ for all\ } z \in\Re^n}.$
    \item (Subdifferential) \label{prop-support-properties-subdiff}
      $\partial\sigma\Cs(z)
       =\set{x\in\Cscr|\ip x z = \sigma\Cs(z)}
       $.
    \item (Recession cones) \label{prop-support-properties-recession}
      $\gauge\Cs(x)=0$ if and only if $x\in\rec\Cscr$.
    \end{propenum}
\end{proposition}
\begin{proof}
\begin{itemize}

\item[(a)] Because $\Dscr \subseteq \cl\conv\Dscr$, it follows that
  $\sigma_{\Dscr}(z) \leq \sigma_{\cl\conv\Dscr}(z)$ for all
  $z$. Hence it is sufficient to prove that
  $\sigma_{\cl\conv\Dscr}(z) \leq \sigma_{\Dscr}(z)$ for all $z$. Fix
  any $d \in \cl\conv\Dscr$ and choose an arbitrary sequence
  $\{d_k\}_{n = 1}^\infty \subset \conv\Dscr$ such that $d_k \to
  d$. Each element of the sequence $\{d_k\}$ is a convex combination of
  points in $\Dscr$, and so it follows that
  $\ip{d_k}{z} \leq \sigma_{\Dscr}(z)$ for all $k$ and $z$. Since
  $d_k \to d$ and $\ip{d_k}{z} \leq \sigma_{\Dscr}(z)$ for all $n$, it
  follows that $\ip{d}{z} \leq \sigma_{\Dscr}(z)$. But $d$ is
  arbitrary, and so we can conclude that
  $\sigma_{\cl\conv\Dscr}(z) \leq \sigma_{\Dscr}(z)$.
	
\item[(b)] The gauge to $\Cscr\polar$ (see~\eqref{eq:4}) can be
  expressed as
  $\gauge\Csp(x) = \inf\set{\lambda>0 | \lambda\inv x \in
    \Cscr^{\circ}}$. Thus, from the definition of the polar
  set~\eqref{eq-polar-set},
  \begin{align*}
    \gauge\Csp(x) &= \inf\set{\lambda>0 | \ip{\lambda\inv x}{y} \leq 1, \ \forall y \in \Cscr}
    \\&= \bigl[\sup\set{\mu>0 | \ip{\mu x}{y} \leq 1,\ \forall y \in \Cscr}\bigr]\inv
    \\&= \bigl[\sup\set{\mu>0 | \ip{x}{y} \leq \mu\inv,\ \forall y \in \Cscr}\bigr]\inv
    \\&= \sup_{y \in \Cscr}\, \ip x y
    = \sigma\Cs(x).
\end{align*} 

\item[(c)]
  From the Minkowski functional expression for the gauge function, 
  \begin{align*}
    \gauge\Cs(Mx) &= \inf\set{\lambda \mid Mx \in \lambda C}
    \\&= \inf\set{\lambda \mid x \in M^{-1}(\lambda C)}
    \\&= \inf\set{\lambda \mid x \in \lambda M^{-1}C} = \gauge_{M^{-1}C}(x).
  \end{align*}
  Also, from the definition of the adjoint of a linear operator,
  \begin{align*}
    \sigma_{M\Cscr}(z) &= \sup\set{\ip{M x}{z}| x \in \Cscr}
    \\&= \sup\set{\ip{x}{M^* z}| x \in \Cscr}
       = \sigma\Cs(M^* z).
  \end{align*}

\item[(d)]
  By defining $M = \alpha$, the proof follows directly from~\cref{prop-linear-transform}. 

\item[(e)]

  It follows from the definition of the domain that
  $\dom\gauge\Cs = \cone\Cscr$. So we only need to show that
  $\dom\sigma\Cs=(\rec\Cscr)\polar$. First we show that
  $\dom\sigma\Cs \subseteq (\rec\Cscr)\polar$. For any
  $x \in \dom\sigma\Cs$, the support $\sigma\Cs(x)$ is finite. Thus
  for any $d \in \rec\Cscr$,
  \[\ip{c + \lambda d}{x} < \infty, \quad \forall c \in \Cscr, \lambda
    \geq 0;\] see~\eqref{eq-recession-cone}. It follows that $\ip d x \leq 0$, and thus
  $x \in (\rec\Cscr)\polar$. For the other direction, instead we will
  show that $(\dom\sigma\Cs)\polar \subseteq \rec\Cscr$. Assume
  $x \in (\dom\sigma\Cs)\polar$, then for any $c \in \Cscr$,
  $\lambda \geq 0$, $y \in \dom\sigma\Cs$, we have
  \[
    \ip{c + \lambda x}{y} = \ip{c}{y} + \lambda\ip{x}{y} \leq
    \ip{c}{y} \leq \sigma\Cs(y).
  \]
  Since $\Cscr$ is a closed convex set, we can conclude that
  $c + \lambda x \in \Cscr$, for all $c \in \Cscr$ and
  $\lambda \geq 0$. Therefore, $x \in \rec\Cscr$.

\item[(f)]Let
  $\Dscr = \set{x\in\Re^n| \ip{x}{z} \leq \sigma\Cs(z) \mbox{\ for
      all\ } z \in\Re^n}$. By the definition of support function, it
  can be easily shown that $\Cscr \subseteq \Dscr$. So we only need to
  prove that $\Dscr \subseteq \Cscr$.  Assume there is some
  $x \in \Dscr$ such that $x \notin \Cscr$. Then by the separating
  hyperplane theorem, there exists $s \in \Re^n$ such that
\[\ip{s}{x} > \sup\set{\ip{s}{y}| y \in \Cscr} = \sigma\Cs(s).\]
This leads to a contradiction. Therefore, we can conclude that $\Cscr = \Dscr$.
	
\item[(g)] Let $\Dscr = \set{x\in\Cscr|\ip x z =
    \sigma\Cs(z)}$. First, we show that
  $\Dscr \subseteq \partial\sigma\Cs(z)$. Assume $x \in \Dscr$, then
  for any $w \in \Re^n$,
  \[\sigma\Cs(w) \geq \ip{x}{w} = \sigma\Cs(z) + \ip{x}{w - z}.\]
  Thus, $x \in \partial\sigma\Cs(z)$. Next, we prove that
  $\partial\sigma\Cs(z) \subseteq \Dscr$. Assume
  $x \in \partial\sigma\Cs(z)$, then
  \begin{equation}\label{eqn:helper}
    \sigma\Cs(w) \geq \sigma\Cs(z) + \ip{x}{w - z}, \quad \forall w \in \Re^n
  \end{equation}
  By the subadditivity of support functions, we must have
  \begin{equation} \label{eq:14}
    \sigma\Cs(z) + \sigma\Cs(w - z) \geq \sigma\Cs(w), \quad \forall w \in \Re^n.
  \end{equation}
  It then follows from~\eqref{eqn:helper} and~\eqref{eq:14} that
  $\sigma\Cs(v) \geq \ip{x}{v}$ for all $v$.  By part (d), we can thus
  conclude that $x \in \Cscr$. Now let $w = 0$ in~\eqref{eqn:helper},
  it follows that $\ip{x}{z} \geq \sigma\Cs(z)$.  Therefore, it
  follows that $\ip{x}{z} = \sigma\Cs(z)$ and thus $x \in \Dscr$.

\item[(h)]

  First, assume $\gauge\Cs(x) = 0$. Then for any $\xhat \in \Cscr$ and
  $\lambda \geq 0$,
  \[
    \gauge\Cs(\hat{x} + \lambda x) \leq \gauge\Cs(\hat{x}) +
    \lambda\gauge\Cs(x) = \gauge\Cs(\hat{x}).
  \]
  It follows that $\xhat + \lambda x \in \Cscr$ and therefore
  $x\in\rec\Cscr$. Next, assume $x\in\rec\Cscr$. Then by the
  definition of recession cone, we have $\lambda x \in \Cscr$ for all
  $\lambda \geq 0$, which implies $\gauge\Cs(x) = 0$.

\end{itemize}	
\end{proof}

\subsection{Exposed faces}

A face $\Fscr\Cs$ of a convex set $\Cscr$ is a subset with the property that for
all elements $x_1$ and $x_2$ both in $\Cscr$, and for all $\theta\in (0,1)$,
\[
  \theta x_1 + (1-\theta) x_2\in \Fscr\Cs \quad \iff\quad x_1\in \Fscr\Cs
  \text{and} x_2\in \Fscr\Cs.
\]
Note that the face must itself be convex. A particular face $\Fscr\Cs(d)$ is
\emph{exposed} by a direction $d\in\Re^n$ if the face is contained in the
supporting hyperplane with normal $d$:
\begin{equation} \label{eq:face} \Fscr\Cs(d) = \set{c\in\Cscr|\ip c d =
    \sigma\Cs(d)} =\partial\sigma\Cs(d),
\end{equation}
where the second equality follows from \cref{prop-support-properties-subdiff}.
The elements of the exposed face $\Fscr\Cs(d)$ are thus precisely those elements
of $\Cscr$ that achieve the supremum for $\sigma\Cs(d)$.

In \cref{sec-convexanalysis-atomic} we will consider atomic sets that are not
convex. In that case, the exposed face of the convex hull of those atoms
coincides with the convex hull of the exposed atoms. In particular, if
$\Ascr=\{a_i\}_{i\in\Iscr}$ is any collection of atoms and
$\Cscr=\conv(\Ascr\cup\{0\})$, then
\[
  \Fscr\Cs(d) = \conv\Escr\As(d).
\]

It follows from positive homogeneity of the support function $\sigma\Cs$ that
\begin{equation}\label{eq:11}
  \Fscr_{\scriptscriptstyle\!\alpha\Cscr}(d)=\alpha\Fscr\Cs(d)
  \text{and}
  \Fscr\Cs(\alpha d)=\Fscr\Cs(d) \quad \forall \alpha>0.
\end{equation}
For nonpolyhedral sets, it is possible that some faces may not be exposed
\cite[p.~163]{rockafellar1970convex}.

\subsection{Alignment characterization}

The definition of alignment in \cref{def:alignment} rests on the tightness of
the polar inequality~\eqref{eq:polar-inequality}. In this section we tie the
alignment condition to a more geometric concept based on exposed faces, which
uncovers the dual relationship between a pair of aligned
vectors. We proceed in two steps. First, we characterize alignment for a pair of
vectors that are in the unit level sets, respectively, for a gauge and its
polar. Second, we generalize this result to any vectors in the respective domains.

\begin{proposition}[Normalized alignment]
  \label{prop-normalized-alignment}
  For any elements $x\in\Cscr$ and $z\in\Cscr\polar$, the following conditions
  are equivalent:
  \begin{propenum}
  \item $\ip x z = 1$,
  \item $z\in\Fscr_{\scriptscriptstyle\Cscr\polar}(x)$,
  \item $x\in\Fscr_{\scriptscriptstyle\Cscr}(z)$.
  \end{propenum}
  Moreover, these statements imply that $x\in\bnd\Cscr$ and
  $z\in\bnd\Cscr\polar$.
\end{proposition}
\begin{proof}
  Suppose that (a) holds. By the definition \eqref{eq-polar-set} of
  the polar set $\Cscr\polar$,
  \[
    \sigma\Csp(x) = \sup\set{\ip x u|u\in\Cscr\polar}\le1
    \quad\forall x\in\Cscr.
  \]
  Then (a) implies that $z$ achieves the supremum above, and so
  by~\eqref{eq:face}, this holds if and only if $z\in\FaceCp(x)$.
  Thus (b) holds. The fact that (b) implies (a) follows by simply
  reversing this chain of arguments.

  To prove that (a) is equivalent to (c), we only need to use the
  assumption that $\Cscr$ is closed and contains the origin, and hence
  that $\Cscr=\Cscr^{\circ\circ}$. This allows us to reuse the
  arguments above by exchanging the roles of $x$ and $z$, and $\Cscr$
  and $\Cscr\polar$.
\end{proof}

The following corollary characterizes the general alignment condition
without assuming that the vector pair $(x,z)$ is normalized.

\begin{corollary}[Alignment] \label{cor-general-alignment} Let
  $x\in\cone\Cscr$ and $z\in\cone\Cscr\polar$ be any two vectors. The
  pair $(x,z)$ is $\Cscr$-aligned if any of the following equivalent
  conditions holds:
  \begin{corenum}
  \item \label{cor-general-pi} $\ip x z = \gauge\Cs(x)\cdot\sigma\Cs(z)$,
  \item $z \in \cone\Fscr\Csp(x) + \rec\Cscr\polar$,
  \item $x \in \cone\Fscr\Cs(z) + \rec\Cscr$.
  \end{corenum}
\end{corollary}
\begin{proof}
  First suppose that $\gauge\Cs(x)$ and $\sigma\Cs(z)$ are
  positive. Then the equivalence of the statements follows by applying
  \cref{prop-normalized-alignment} to the normalized pair of vectors
  $\xhat:=x/\gamma\Cs(x)$ and $\zhat:=z/\sigma\Cs(z)$. In that case
  Part 1 follows immediately after multiplying
  $\ip{\xhat}{\zhat\pthinsp}=1$ by the quantity
  $\gauge\Cs(x)\cdot\sigma\Cs(z)$. Parts 2 and 3 follow from the fact
  that for any convex set $\Dscr$ and any vector $d$,
  $\Face\Dscr(d)=\Face\Dscr(\alpha d)$ for any positive scalar
  $\alpha$; see~\eqref{eq:11}.

  We now show equivalence of the statements in the case where
  $\gauge\Cs(x)=0$. By \cref{prop-support-properties-recession}, this
  holds if and only if $x\in\rec\Cscr$, but not in $\Face\C(z)$. Thus
  Part 3 holds. But because $\sigma\Cs(z)$ is finite, $x$ and $z$
  together satisfy $\ip x z =0$. Thus, Part 1 holds. To show that Part
  2 holds, notice that $\sigma\Csp(x)=\gauge\Cs(x)=0$, and so
  by~\eqref{eq:face},
  \[
    \cone\Face\Csp(x) = \set{ u |  \ip x u = 0},
  \]
  which certainly contains $z$. Thus, Part 2 holds.  The case with
  $\sigma\Cs(z)=0$ follows using the same symmetric argument used in
  the proof of \cref{prop-normalized-alignment}.
\end{proof}

Relative to~\cref{prop-normalized-alignment}, this last result is most
interesting when one of the elements in the aligned pair $(x,z)$
belongs to the recession cones of $\Cscr$ or its polar
$\Cscr\polar$. In that case, the alignment condition in
\cref{cor-general-pi} requires the vectors to be orthogonal,
i.e., $\ip x z=0$. But if $x\in\rec\Cscr$, the requirement that
$z$ is in the polar $(\rec\Cscr)\polar$ implies that $x$ and $z$ are
extreme rays of their respective recession cones that are orthogonal
to each other. This situation is illustrated in
\cref{fig-recession-example}.

\begin{figure}[t]
  \centering
  \begin{tabular}{@{}c@{\hspace{.5in}}c@{}}
    \includegraphics[page=2]{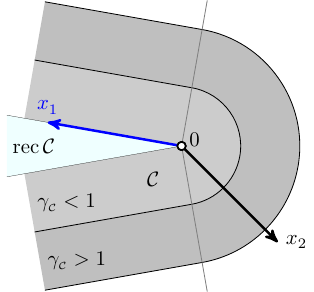}
    & \includegraphics[page=3]{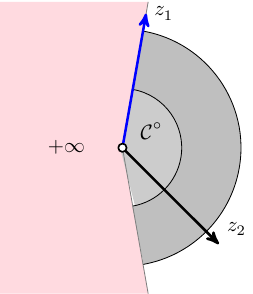}
      \end{tabular}
      \caption{The contours of the gauge function of $\Cscr$ (left)
        and of $\Cscr\polar$ (right). All vectors $x$ in the recession
        cone of $\Cscr$ have gauge value $\gauge\Cs(x)=0$. A vector
        $x_1$ can only be $\Cscr$-aligned with another vector $z_1$ if
        they are orthogonal to each and each is an extreme ray,
        respectively, of $\rec\Cscr$ and
        $\dom\gamma\Csp=(\rec\Cscr)\polar$. Each of the pairs $(x_1,z_1)$ and $(x_2,z_2)$ are
        $\Cscr$-aligned.}
  \label{fig-recession-example}
\end{figure}

\begin{example}[Convex cones]
  \label{eq:convex-cones}
  Suppose that $\Cscr=\Kscr$ is a cone. Because a cone is its own
  recession cone, $\rec\Kscr=\Kscr$. Then for any pair $(x,z)$ that is
  $\Kscr$-aligned, \cref{cor-general-alignment} asserts
  \[
  \ip x z = 0      \quad\Longleftrightarrow\quad
  x\in\Kscr        \quad\Longleftrightarrow\quad
  z\in\Kscr\polar.
\]
This assertion effectively generalizes~\cref{example-linear-subspaces}, which
made the same assertion for linear subspaces.

For convex cones, we thus see that alignment is equivalent to
orthogonality. This principle applies to general convex sets $\Cscr$
using the lifting technique described in \cref{sec-polarity}. Take any
pair $(x,z)\in\Cscr\times\Cscr\polar$ that is $\Cscr$-aligned, which
implies $\ip x z=1$. Then $\xbar:=(x,1)\in\Kscr\Cs$ and
$\zbar:=(z,-1)\in\Kscr\Csp$, and
\[
  \ip*{\xbar}{\zbar} = \ip x z -1 = 0.
\]
This coincides with tightness of the inequality~\eqref{eq:10}, which
characterizes polarity of cones.
\end{example}

The next example shows how the alignment property is connected to
complementarity in conic programming~\cite[Section 5.3.6]{bertsekas2009convex}. Section~\ref{sec-manifestations} explores a more
general connection between alignment and optimality in convex optimization.
\begin{example}[Alignment as optimality in conic optimization] \label{example-conic-opt}

Consider the pair of dual linear conic optimization problems
    \begin{equation*}
      \begin{array}{l@{\enspace}l}
        \minimize{x} & \ip c x \\ 
        \st          & Fx = b,\  x\in \Kscr,
      \end{array}
      \qquad
      \begin{array}{l@{\enspace}l}
        \maximize{y,\,z} & \ip b y\\ 
        \st            & F\T y - z = c, \ z\in \Kscr\polar,
      \end{array}
    \end{equation*}
    where $F:\Re^n\to\Re^m$ is a linear operator,
    $(b,c)\in\Re^m\times\Re^n$ are arbitrary vectors, and
    $\Kscr\polar$ is the polar cone of $\Kscr$.

    The feasible triple $(x,y,z)$ is optimal if strong
    duality holds, i.e.,
    \begin{align*}
      0 = \ip c x - \ip b y = \ip{F\T y-z}{x} - \ip{Fx}{y} = \ip x z.
    \end{align*}
    But because $x\in\Kscr$ and $z\in\Kscr\polar$, it follows
    from~\cref{eq:convex-cones} that $x$ and $z$ are $\Kscr$-aligned.
\end{example}
  
\subsection{Alignment as orthogonal decomposition}

The Moreau decomposition for cones
\cite[Theorem~3.2.5]{hiriart-urruty01} can be used to separate an
arbitrary vector into components that are aligned with respect to any
convex set $\Cscr$.

  Every element, respectively, in $\Kscr\Cs$ and $\Kscr\Csp$ is a
  nonnegative multiple of $(x,1)$ and $(z,1)$ for some vectors
  $x\in\Cscr$ and $z\in\Cscr\polar$. Thus, for any vector
  $(s,\alpha)\in\Re^n\times\Real$, Moreau's decomposition implies
  unique nonnegative scalars $\alpha_x$ and $\alpha_z$ such that
  \begin{align*}
    (s,\alpha) = \proj_{\Kscr\Cs}(s,\alpha)
                 + \proj_{\Kscr\Csp}(s,\alpha)
             = \alpha_x (x,1) + \alpha_z(z,-1).
  \end{align*}
  Orthogonality of the decomposition implies
  $\alpha_x\cdot \alpha_z (\ip x z - 1) = 0$. Then the pair of vectors
  $\xhat = \alpha_x x$ and $\zhat=\alpha_z z$ are $\Cscr$-aligned
  because
  \[
    \ip{\xhat}{\zhat} = \alpha_x\cdot\alpha_z
  \]
  and their corresponding gauge and support values are
  $\alpha_x = \gauge\Cs(\xhat)$ and $\alpha_z = \sigma\Cs(\zhat)$. See
  \cref{fig:moreau}.

  \begin{figure}
    \centering
    \hspace{3.5cm}\includegraphics[page=8]{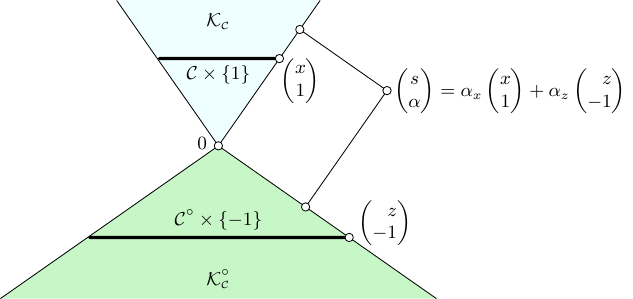}
    \caption{Any vector $(s,\alpha)\in\Re^{n}\times\Re$ can be
      decomposed into orthogonal components in the cones generated by
      a convex set $\Cscr\subset\Re^n$ and its polar. The components
      of the decomposition $s=\alpha_xx+\alpha_zz$ are
      $\Cscr$-aligned.}
    \label{fig:moreau}
  \end{figure}

\section{Alignment with respect to atomic sets} \label{sec-convexanalysis-atomic}

The discussion in \cref{sec-convexanalysis} applies to any convex set, and used
the symbol $\Cscr$ to refer to such general sets. We now turn our attention to
properties of gauges and support functions generated by atomic sets
$\Ascr\subset\Real^n$ not necessarily convex. As we did with~\eqref{eq:4}, we thus adopt the notation
\[
  \gauge\As:=\gauge_{\scriptscriptstyle\convA},
  \quad \sigma\As:=\sigma_{\scriptscriptstyle\convA},
  \quad \FaceA:=\Fscr_{\scriptscriptstyle\convA}.
\]

\subsection{Atomic decomposition}

Two different expressions are given in~\eqref{eq:2} and~\eqref{eq:4}
for a gauge function $\gauge\As$. The next result highlights
the decomposition implicit in the Minkowski functional
to an atomic set. This decomposition allows us to establish other
useful results that involve only the atomic set, rather than its
convex hull, as in \cref{sec-convexanalysis}. This equivalence is used by
Chandrasekaran~\cite{chandrasekaran2012convex} and others. 

\begin{proposition}[Gauge equivalence]
  \label{prop-guage-equivalence}
  For any set $\Ascr\subset\Re^n$ and $\convA:=\conv(\Ascr\cup\{0\})$,
  the following equivalence holds:
  \begin{equation} \label{eq:5}
    \gauge\As(x)
    :=\inf\set{\lambda\ge0|x\in\lambda\convA}
    =\inf_{c_a} \Set{\sum_{a\in\Ascr}c_a |
      x = \sum_{a\in\Ascr}c_a a,\ c_a\ge0\ \forall a\in\Ascr}.
  \end{equation}
\end{proposition}
\begin{proof}
  Take any $x\in\cone\convA$, since otherwise the sets above are
  empty, and by convention, both expressions have infinite
  value. Then, because we can exclude a convex
  combination of the elements of $\convA$,
  \begin{align*}
    \inf\set{\lambda\ge0 | x\in\lambda\convA}
    &= \inf_{\lambda,\,\bar c_a}\biggl\{\lambda\ge0 \biggm| x = \lambda\sum_{a\in\Ascr}\bar c_aa,\ \sum_{a\in\Ascr}\bar c_a=1,\ \bar c_a\ge0\ \forall a\in\Ascr\biggr\}
  \\&=  \inf_{\lambda,\,c_a}\biggl\{\lambda \biggm| x = \sum_{a\in\Ascr}c_aa,\ \sum_{a\in\Ascr}c_a=\lambda,\ c_a\ge0\ \forall a\in\Ascr\biggr\},
  \end{align*}
  which, after eliminating $\lambda$, yields the required equivalence
  shown in~\eqref{eq:5}.
\end{proof}
Some atomic sets, such as the set of rank-1 outer products used to
define the nuclear-norm ball (cf. \cref{example-nuclear-norm}), may be
uncountably infinite. However, when $x\in\cone\convA$, the gauge value
is always finite and the sum $\sum_{a\in\Ascr}c_a$ necessarily converges to a
finite value. This ``sum form'' of the gauge function is useful because it
provides a ``one-norm-like'' interpretation of gauges in terms of the
minimal conic decomposition $\{c_a a\}_{a\in\Ascr}$, which further
suggests that gauges are the natural promoters of atomic sparsity.

\begin{proposition}[Finite support]
  For any point $x\in\cone\convA$, a finite support set $\suppa(x)\in\supp\As(x)$ always exists.
\end{proposition}
\begin{proof}
  If $\gauge\As(x)=0$, the assertion is trivially true, since the
  empty set is the only element of $\supp\As(x)$. Now suppose
  $\gauge\As(x)>0$, and define the normalized vector
  $\xhat=x/\gauge\As(x)$. Then $\xhat\in\cl\conv\Ascr$, and
  $\gauge\As(x)=1$. By Carath\'eory's Theorem
  \cite[Theorem~17.1]{rockafellar1970convex}, there exists a finite
  convex decomposition of $\xhat$ in terms of at most $n+1$ atoms in
  $\Ascr$. That is, there exists a set $\Sscr\subset\Ascr$ with $n+1$
  elements such that
  \[
    \xhat = \sum_{a\in\Sscr}\chat_aa,
    \quad
    \sum_{a\in\Sscr}\chat_a = 1,
    \quad
    \chat_a>0,\ \forall a\in\Sscr.
  \]
  Taking $c_a=\gauge\As(x)\chat_a$ for each $a\in\Sscr$ gives a solution  to the equations in \eqref{eq:3}, showing that $\Sscr\in\supp\As(x)$.
\end{proof}

The support may not be unique, even if it is minimal, e.g., there
exist no other supports with smaller cardinality.
  
\begin{example}[Non-uniqueness]
  Consider the atomic set $\Ascr=\set{(\pm1,\pm1,1)}\subset\Re^3$. The
  point $x=(0,0,2)$ can be expressed in at least three different ways,
  \begin{align*}
    x &= (1,1,1) + (-1,-1,1)
    \\&= (1,-1,1) + (-1,1,1)
    \\&= \half[(1,1,1) + (-1,-1,1) + (1,-1,1) + (-1,1,1)],
  \end{align*}
  all of which give gauge value $\gauge\As (x)=2$. In this case, we write
  \[
    \supp\As\left(\bmat{ 0 \\ 0 \\ 2 } \right) = 
    \left\{
      \left\{\! 
        \bmatr{ 1 \\ 1 \\ 1}, 
        \bmatr{-1 \\ -1 \\ 1}
      \!\right\},
      \,
    \left\{\! 
      \bmatr{1 \\ -1 \\ 1},
      \bmatr{-1 \\ 1 \\ 1}
    \!\right \},
    \,
    \half\left\{\! 
      \bmatr{ 1 \\  1 \\ 1},
      \bmatr{-1 \\ -1 \\ 1},
      \bmatr{ 1 \\ -1 \\ 1},
      \bmatr{-1 \\  1 \\ 1}
    \!\right \}\right\}.
\]
Any element of $\supp\As(x)$ is a valid support set of $x$ with
respect to the atomic set $\Ascr$.  However, for functions commonly
used to promote sparsity, often the support set is always unique.
\end{example}

\cref{prop-guage-equivalence} establishes that the gauge value
$\gauge\As(x)$ of a vector $x$ yields a conical decomposition whose
coefficient sum is minimal. If another vector $v$ can be conically
decomposed as a subset of the atoms of $x$, then the support for $v$
is a subset of the support of $x$, i.e.,
$\suppa(v)\subset\suppa(x)$. This is established in the following
proposition.

\begin{proposition}[Same support sets] \label{prop:same-support-sets}
  Suppose that $\Sscr_\Ascr(x)\subseteq \Ascr$ is a support set for
  some vector $x\in\cone\convA$ with $\gamma(x)>0$.  Then any vector $v$ that has a valid
  conic decomposition in terms of the support $\Sscr\As(x)$, i.e.,
  \begin{equation}\label{eq:28} 
    v = \sum_{\ \mathclap{a\in \suppa(x)}}c_a a, \quad c_a \geq 0,
  \end{equation}
  has gauge value
  \[
    \gauge\As(v) = \sum_{\ \mathclap{a\in \suppa(x)}} c_a. 
  \] 
\end{proposition}
\begin{proof}
  Suppose, by way of contradiction, that there exists a conic decomposition of
  $v$ with respect to $\Ascr$ that is not given by~\eqref{eq:28}, i.e.,
  \[
    v = \sum_{a\in \Ascr} c'_a a,\qquad c_a' \geq 0, \qquad
    \sum_{a\in\Ascr} c'_a < \sum_{\mathclap{a\in\Sscr\As(x)}}c_a.
  \] 
  Because $\suppa(x)$ is the support set of
  $x$, there exist positive coefficients $\chat_a$ where
  \[
          x = \sum_{\mathclap{a\in\suppa(x)}} \chat_a a,
          \qquad
          \gauge\As(x) = \sum_{\mathclap{a\in \suppa(x)}} \chat_a.
  \]
  But a valid decomposition of $x$ is 
  \[
          x = \beta v + x - \beta v = \beta \sum_{a\in \Ascr} c'_a a +
          \sum_{a\in \suppa(x)} (\chat_a - \beta c_a) a,
  \]
  where we pick
        $\beta = [\min_{a\in \suppa(x)} \chat_a]/[\max_{a\in
          \suppa(x)} c_a]$ to guarantee that all the coefficients are
        nonnegative. 
         Then by definition of gauges,
  \[
          \sum_{a\in \suppa(x)} \chat_a \leq \beta \sum_{a\in\Ascr}
          c'_a + \sum_{a\in \suppa(x)} (\chat_a - \beta c_a) \iff
          \sum_{a\in\Ascr} c'_a \geq \sum_{a\in\suppa(x)} c_a.
  \]
  This implies that the decomposition of $v$ with respect to
        $\suppa(x)$ is in fact the \emph{minimal} decomposition of $v$ with respect
        to $\Ascr$, and the sum of the coefficients indeed giving its gauge value.  
\end{proof}

\begin{proposition}[Support identification]
  For any set $\Ascr\subset\Re^n$, the elements $x$ and $z$ in $\Re^n$
  are $\Ascr$-aligned if and only if
  $\Sscr_\Ascr(x) \subseteq \Escr_\Ascr(z)$ for all
  $\Sscr_\Ascr(x)\in \supp_\Ascr(x)$.
\end{proposition}
\begin{proof}
First, we show that if $x$ and $z$ are $\Ascr$-aligned, then $\Sscr_\Ascr(x) \subseteq \Escr_\Ascr(z)$ for all
  $\Sscr_\Ascr(x)\in \supp_\Ascr(x)$.  Because the elements $x$ and $z$ are $\Ascr$-aligned,
  \begin{equation}
    \ip x z = \gamma\As(x)\cdot\sigma\As(z).
    \label{eq:proof-helper-1}
  \end{equation}
  Now suppose that $\gamma\As(x) > 0$. Then all support sets
  $\Sscr\As(x) \in \supp\As(x)$ are nonempty. Suppose that
  $a\in \Sscr\As(x)$ but $a\not \in\Escr\As(z)$. We will show that
  this leads to a contradiction. By definition, $a\not\in\Escr\As(z)$
  implies that
  \begin{equation}
    \ip a z < \sigma\As(z).
    \label{eq:proof-helper-2}
  \end{equation}
  Define $v = x - c_a a$, which is the vector that results from
  deleting the atom $a$ from the support of $x$. Then by
  \cref{prop:same-support-sets},
  \begin{equation} \label{eq:12}
    \gamma\As(v) = \gamma\As(x) - c_a.
  \end{equation}
  Thus,
  \[
    \ip x z \overset{(a)}{=} \underbrace{\langle v , z\rangle}_{\leq
      \gamma\As(v) \sigma\As(z) } + \underbrace{c_a\langle a,z\rangle
    }_{<\ c_a \sigma\As(z)} \overset{(b)}{<} (\gamma\As(v) + c_a)
    \sigma\As(z) \overset{(c)}{=} \gamma\As(x)\cdot \sigma\As(z)
  \]
  where (a) follows by construction ($x = v + c_a a$); (b) follows
  from the polar inequality \eqref{eq:polar-inequality} and assumption
  \eqref{eq:proof-helper-2}; and (c) follows from \eqref{eq:12}. But
  this contradicts \eqref{eq:proof-helper-1}, and therefore
  $a\in\Sscr_\Ascr(x)$ implies $ a\in \Escr_\Ascr(z)$, i.e.,
  $\suppa(x)\subseteq\Escr\As(z)$.
  
  Now assume  $\gauge\As(x)=0$. Then
  $x\in\rec\convA$ and $\supp\As (x)$ contains only the empty set. Since empty sets are also a subset of $ \Escr_\Ascr(z)$for any $z$, the statement is trivially true. 

  Next, we show that if $\Sscr_\Ascr(x) \subseteq \Escr_\Ascr(z)$ for all
  $\Sscr_\Ascr(x)\in \supp_\Ascr(x)$, then $x$ and $z$ are $\Ascr$-aligned. By the definition of support set~\ref{eq:3}, we can assume that 
  \[
    \gauge\As(x) = \sum_{\mathclap{a\in\suppa(x)}} c_a,
    \qquad x = \sum_{\mathclap{a\in\suppa(x)}} c_aa,
    \qquad c_a > 0\ \forall a\in\suppa(x).
  \]
  Then by~\cref{cor-general-alignment}, we only need to show that $\ip{x}{z} = \gauge\As(x)\sigma\As(z)$. Indeed, 
  \begin{align*}
    \ip{x}{z} &= \sum_{\mathclap{a\in\suppa(x)}}c_a\ip{a}{z}
    \\&= \bigg(\sum_{\mathclap{a\in\suppa(x)}}c_a\bigg)\sigma\As(z)
    = \gauge\As(x)\sigma\As(z),
  \end{align*}
  where the second equality follows from the assumption that $\Sscr_\Ascr(x) \subseteq \Escr_\Ascr(z)$. 
\end{proof}

\subsection{Examples}

The general alignment result described by \cref{cor-general-alignment} includes
the possibility that aligned vectors may contain elements from the recession
cone of the atomic set. Elements in the recession cone may be interpreted as
directions, rather than just points in the set. The presence of a non-trivial
recession cone must be considered in practice, and is exhibited, for example, by
all seminorms: these are nonnegative functions that behave like norms with the
exception that they may be zero at nonzero points and are not necessarily
symmetric. The next example describes a common atomic set that is composed by
points and directions.

\begin{example}[Total variation]
  The anisotropic total-variation norm of an $n$-vector $x$ is defined as
\[
  \|x\|_{\scriptscriptstyle\rm TV}
  = \sum_{i=2}^n |x_i-x_{i-1}| = \|Dx\|_1,\qquad
  D = \begin{bmatrix*}[r]
    1 & -1 & 0 & \cdots & 0& 0 \\
    0 &1 & -1 &  \cdots & 0 & 0\\
    \vdots &\vdots & \vdots & \ddots & \vdots  & \vdots\\
    0 & 0 & 0 & \cdots & 1 & -1
    \end{bmatrix*}.
\]
The matrix $D$ has a 1-dimensional nullspace spanned by the constant vector of
all ones $e$, and so $De=0$. The TV norm is then
a seminorm, and thus the atomic set must include a direction
of recession, given by the range of $e$.  Interestingly, the atomic
set that induces this norm is not unique: for any matrix
$A=[a_1,\ldots,a_{n-1}]$ where $DA = I$, the corresponding TV norm is the gauge with
respect to the atoms
\begin{equation*}
\Ascr = \{\pm a_1,\ldots,\pm a_{n-1} \} + \cone (\pm e).
\end{equation*}
To see this, write
\[
  x = \sum_{i=1}^{n-1} c_i a_i + c_e e = A c + c_e e
\]
for some scalars $c_1,\ldots,c_{n-1}$ and $c_e$. (The scalars are not
restricted to be nonnegative because the set of atoms includes vectors
with both positive and negative signs.)  Note that the $n-1$ vectors
$a_i$ span $\Null(e)$, so the above decomposition always exists, with
unique values for $c_i$ and $c_e$. The solution to \eqref{eq:3} thus
determines the unique decomposition
\[
  x = \sum_{i=1}^{n-1} \underbrace{(s_ic_i)}_{c_a}\cdot\underbrace{(s_ia_i)}_{a} + c_e e, \quad s_i = \sign(c_i),
\]
where $(s_ic_i)$ are the coefficients for the atoms
$(s_ia_i)\in\Ascr$, and $c_e$ is the coefficient for the recession
direction $e$.  Then
\[
\|Dx\|_1 = \|DAc\|_1 = \|c\|_1 = 
\gauge\As(x).
\]
If $x\in\cone(\pm e)=\rec\convA$, then $\gauge\As(x)=0$.

To see that the atomic set is not unique, note that $DA = I$ for
any matrix of the form $A = B+es^T$, where 
\begin{equation*}
B = [b_1,\ldots,b_{n-1}] := \begin{bmatrix}
1 & 1 & \cdots & 1&1\\
0 & 1  & \cdots & 1&1 \\
\vdots  & \vdots & \ddots & \vdots  &\vdots\\
0 & 0  & \cdots & 1 & 1\\
0 & 0  & \cdots & 0 & 1\\
0 & 0  & \cdots & 0 & 0
\end{bmatrix}
\end{equation*}
and $s\in \Re^{n-1}$ is an arbitrary vector.  However, the gauge
function with respect to the atomic set formed by the columns of $B$ and $e$ 
is well defined. Specifically, note that the range of the matrix $[B\ e]$
spans all of $\Re^n$. Thus the decomposition
\begin{equation}\label{eq:7}
  x = Bc+c_ee
\end{equation}
uniquely defines the vector $c$ and the scalar $c_e$, and
$\gauge\As(x) = \|Dx\|_1 = \|c\|_1$, as before.

The support function for this set of atoms is
\begin{align*}
  \sigma\As(z)
  &= \sup\set{\ip x z | x = Bc+c_ee,\ \infnorm{c}\le1}
\\&= \sup\set{\ip{c}{B\T z} + c_e\ip e z | \infnorm{c}\le1 }.
\end{align*}
Note that if $z\not\in\Null(e)$, then $\sigma\As(z)$ clearly
unbounded because $c_e$ is not constrained. This confirms the fact
that the domain of $\sigma\As$ is $(\rec\convA)\polar=\Null(e)$, as
shown by \cref{prop-support-properties-domain}.
\cref{cor-general-alignment} asserts that if $z$ is $\Ascr$-aligned
with $x$, then it exposes all of the atoms that contribute
non-trivially towards the decomposition~\eqref{eq:7}. In particular,
$\suppa(x)\subset\Escr\As(z)$, where one such decomposition gives
\[
  \suppa(x) = \set{ \sign((Dx)_i) b_i | x_i \ne 0 }
  \text{and}
  \Escr\As(z) = \set{\sign(\ip{b_i}{z}) b_i | \max_j\; \ip{b_j}{z} = \abs{\ip{b_i}{z}} }.
\]
Notice that these alignment conditions do not depend on the specific
choice of the representation $\Ascr$, and are defined only with
respect to the columns of $B$, which are fixed.
\end{example}

Group norms arise in applications where the nonzero entries of a
vector are concentrated in patters across the vector. Applications
include source localization, functional magnetic resonance imaging, and others
 \cite{cottraoengakreu:2005,GOR1995GRa,jacob2009group}. One interesting feature
of group norms is that they are not polyhedral. 
\begin{example}[Group norms]
  Consider the $\ell$ subsets $g_i\subseteq\{1,\ldots,n\}$ such
  $\cup_{i=1}^\ell g_i=\{1,\ldots,n\}$. Define the \emph{group norm}
  with respect to the groups $\Gscr=\{g_1,\ldots,g_\ell\}$ as the
  solution of the convex optimization problem
  \begin{equation} \label{eq:8}
    \norm{x}_\Gscr=\min_{y_i}
    \left\{
      \sum_{i=1}^\ell\norm{y_i}_2 \Bigm| x = \sum_{i=1}^\ell P_{g_i}y_i
    \right\},
  \end{equation}
  where the linear operator $P_\Iscr:\Re^{|\Iscr|}\to\Re^n$ scatters
  the elements of a vector into an $n$ vector at positions indexed by
  $\Iscr$, i.e., $\{(P_\Iscr y)_i\}_{i\in\Iscr}=y$, and
  $(P_\Iscr y)_k=0$ for any $k\notin\Iscr$. This norm is induced by
  the atomic set
  \[
    \Ascr = \left\{
      P_{g_i}s_i \Bigm| s_i\in\Real^{|g_i|},\ \norm{s_i}_2=1,\ i=1,\ldots,\ell
    \right\},
  \]
  which yields the decomposition
  \begin{equation}\label{eq:9}
    x = \sum_{i=1}^\ell
    c_i (P_{g_i}s_i),
  \end{equation}
  where $c_i$ and $(P_{g_i}s_i)$ are, respectively, the coefficients and atoms
  of the decomposition.
  
  If the sets in $\Gscr$ form a partition of $\{1,\ldots,n\}$ then the
  (non-overlapping) group norm is simply
  \[
    \|x\|_\Gscr = \sum_{i=1}^\ell\|x_{g_i}\|_2.
  \]
  A common example is the matrix (1,2) norm, which is the sum of the
  Euclidean norms of the columns of a matrix \cite{ding2006r}. In the
  non-overlapping group case, the support set is unique, and for all
  $i=1,\ldots,\ell$, the coefficients and atoms of the
  decomposition~\eqref{eq:9} are given by
  \[
    c_i=\norm{x_{g_i}}_2, \textt{and} (P_{g_i}s_i) \text{with} s_i=(c_i)^{-1}x_{g_i}.
  \]
  More generally, the support sets $g_i$ may overlap, and thus the
  gauge value of $x$ must be obtained as the solution of the convex
  optimization problem~\eqref{eq:8}.

  The conditions under which a vector $z$ is $\Ascr$-aligned with $x$
  is similar to the 1-norm case. We first decompose by each group
  $g_i$:
  \[
    \sup_{x\in\Ascr}\,\ip x z
    \overset{\rm(a)}=
    \max_{i=1,\ldots,\ell}\ 
      \sup \left\{\ip{s_i}{z_{g_i}} \bigm| \norm{s_i}_2\le1,\ s_i\in\Re^{|g_i|}\right\}
    \overset{\rm(b)}=
      \max_{i=1,\ldots,\ell}\ \norm{z_{g_i}}_2,
  \]
  where (a) follows from applying the supremum to each atom in $\Ascr$ and (b) follows from the definition of the 2-norm. That is to say, $x$ is $\Ascr$-aligned with $z$ if the decomposition~\eqref{eq:9} has $\suppa(x)\subset\Escr\As(z)$, where
  \[
    \suppa(x) = \left\{
      P_{g_i}x_{g_i}/\norm{x_{g_i}}_2 \bigm| \norm{x_{g_i}}_2>0
    \right\}
    \text{and}
    \Escr\As(z) = \left\{
      z_{g_i}/\norm{z_{g_i}}_2
      \bigm|
      \norm{z_{g_i}}_2 = \textstyle\max_j\,\norm{z_{g_i}}_2
    \right\}.
  \]
  
\end{example}

The next two examples are for gauges that encourage sparsity (i.e.,
low-rank) for matrices.

\begin{example}[Trace norm for semidefinite matrices] \label{example-trace-norm}

  An important gauge function is generated by the spectrahedron
  \[\Ascr = \set{uu^T | u\in\Re^n,\ \|u\|_2 = 1},\] which is a subset of the
  nuclear-norm ball that only includes symmetric rank-1 matrices. As with the
  nuclear-norm, this gauge
  encourages sparsity with respect to the set of rank-1
  matrices---i.e., low-rank---and only admits positive definite matrices.

  We first derive the support function with respect to $\Ascr$:
  \[
    \sigma\As(Z) = \sup_{X\in\convA}\,\ip{X}{Z} =
    \max\Biggl\{
      0,\ \sup_{\|u\|_2 = 1} \ip{u}{Zu}
      \Biggr\} =
    \max\{0,\,\lambda_{\max}(Z)\},
  \]
  which vanishes only if $Z$ is negative semidefinite, and otherwise
  is achieved when $u$ is a maximal eigenvector of $Z$.  Let
  $X=U\Lambda U^T$ be the eigenvalue decomposition of $X$. Using
  \cref{prop-support-properties-polarity} together with~\eqref{eq:15},
  which gives us $\Ascr\polar=\set{z|\sigma\As(z)\le1}$, the gauge
  function can be expressed as the support function over $\Ascr^\circ$:
  \begin{align*}
    \gauge\As(X)
    &= \sup\,\set{\ip X Z | \lambda_{\max}(Z) \leq 1}
  \\&= \sup\,\set{\ip{U\Lambda U^T}{Z} | \lambda_{\max}(Z) \leq 1}
  \\&= \sup\,\set{\ip{\diag(\Lambda)}{\diag(U\T Z U)} | \lambda_{\max}(Z) \leq 1},
  \\&= \trace(\Lambda) + \delta_{\succeq0}(X),
  \end{align*}
  where the last equality holds because the supremum is achieved by
  $Z = UU^T$.  The indicator on the semidefinite cone arises because
  the supremum is infinite if any component of $\Lambda$ is
  negative. In other words, indefinite matrices cannot be conically
  decomposed with respect to the atomic set $\Ascr$, which is
  indicated by the infinite value of the gauge.  Moreover, it follows
  that the nontrivial eigenvectors provide a support set for $X$,
  i.e.,
  \begin{equation*} 
    \Sscr\As(X) \supseteq \{u_1u_1^T,\ldots ,u_ru_r^T\},
  \end{equation*}
  where $r$ is the rank of $X$.

  This support is not unique, however, and in fact the set of supports
  of $X$ is very large.  To see this, consider any valid conic atomic
  decomposition
  \[
    X = c_1v_1v_1^T + \cdots + c_kv_kv_k^T = VCV^T,
  \]
  where $c_i$ and $v_i$, respectively, are the $i$th diagonal entry of
  the diagonal matrix $C$ and $i$th column of the matrix $V$. Then
  \[
    \trace(X) = \trace(VCV^T) = \trace(CV^TV)
    = \ip{\diag(C)}{\diag(V^TV)} = \sum_{i=1}^kc_i,
  \]
  where the last equality follows from the fact that each $v_i v_i^T$
  is in $\Ascr$ and thus has unit norm.  Therefore any conic atomic
  decomposition of $X$ yields the same gauge value, which is the trace
  of $X$. Specifically, the support of $X$ with respect to the
  spectrahedron $\Ascr$ can be characterized as
  \begin{equation*}
    \Sscr\As(X) =
    \set{v_1v_1^T,\ldots,v_kv_k^T | \|v_i\|_2 = 1, \;
      \range(V) = \range(X)}.
  \end{equation*}
  Because we do not impose orthonormality among the vectors $v_i$,
  this set is not unique.

  According to~\eqref{eq:13}, the essential atoms are given by the
  eigenvectors corresponding to the maximal eigenvalue of $Z$,
  including all of their convex combinations:
  \[
    \Escr\As(Z) =
    \conv\set{uu^T | u^TZu = \lambda_{\max}(Z)}.
  \]
  This set coincides with the exposed face $\Fscr\As(Z)$; cf.~\eqref{eq:11}.
\end{example}
\begin{example}[Weighted trace norm for semidefinite matrices]

  \label{example-weighted-trace-norm} We describe a generalization of
  the trace norm for positive semidefinite matrices, which was covered
  by \cref{example-trace-norm}. The weighted trace norm is given by
  the function
  \[
    \kappa(X) = \ip L X + \delta_{\succeq0}(X),
  \]
  where $L$ is positive semidefinite. Write the decomposition of $L$
  as
  \[
    L = \bmat{V & \Vbar}\bmat{\Lambda & 0\\0 & 0}\bmat{V^T\\\Vbar^T}
      = V \Lambda V^T,
  \]
  where $\Lambda$ is diagonal with strictly positive elements and $V$
  and $\Vbar$, respectively, span the range and nullspace of $L$.

  We claim that $\kappa$ is the gauge to the atomic set
  \begin{equation} \label{eq:weighted-atomic-set}
    \Ascr = \set{rr^T | r = Vp,\ p\T\Lambda p=1}
          + \set{ss^T | s = \Vbar q \mbox{ for all } q},
   \end{equation}
   which we establish by showing that $X\in\Ascr$ implies
   $\kappa(X)=1$, and vice versa.
  
   Take any element $X\in\Ascr$, and observe
   \[
     \kappa(X) = \ip L X = \ip{L}{Vpp^T V^T} = p\T V\T L V p = p\T\Lambda p = 1.
   \]
   Conversely, take any $X$ such that $\kappa(X)=1$. Then, $X$ is
   PSD. The orthogonal decomposition of $X$ onto the range and
   nullspace of $L$ is given by
   \[
     X = VV^T X VV^T + \Vbar\Vbar^T X \Vbar\Vbar^T.
   \]
   Then,
   \begin{align*}
     1=\kappa(X) = \ip L X
     = \ip L {VV^T X VV^T}
     = \ip{\Lambda}{V\T X V},
   \end{align*}
   which implies that
   $V\T X V \in \conv\set{pp^T| p\T\Lambda p = 1 }$.  Therefore, $X$
   is in the convex hull of $\Ascr$.  The second set in the
   sum~\eqref{eq:weighted-atomic-set} is in the nullspace of $L$ and
   thus can be ignored. This establishes the claim, and also provides
   an expression for the support set to $X$:
   \[
     \Sscr\As(X) = \set{(Vp_i)(Vp_i)^T | p_i\T\Lambda p_i=1,\ \range(V[p_1\cdots p_k])=\range(X)}.
   \]
   The minimal set of vectors needed to complete the support
   is equal to the rank of $X$.

   The support function with respect to $\Ascr$ can be reduced to a
   maximum generalized eigenvalue problem, as follows:
   \begin{align*}
     \sigma\As(Z)
       &= \sup\set{\ip{X}{Z}|X\in\convA}
     \\&= \sup\set{\ip{Vpp\T V\T}{Z} | p\T\Lambda p \le 1}
     \\&= \sup\set{\ip{V\T Z V}{\Lambda^{-1/2}uu\T\Lambda^{-1/2}}|u\T u \le 1}
     \\&= \sup\set{\ip{\Lambda^{-1/2}V\T Z V\Lambda^{-1/2}}{uu^T}|u\T u \le 1}
     \\&= \max\left\{ 0, \lambda\submax(\Lambda^{-1/2}V\T Z V\Lambda^{-1/2})\right\}.
   \end{align*}
   We recognize that the expression inside the supremum is the
   generalized eigenvalue of the pencil $(Z,L)$, so that
   \[
     \sigma\As(Z) = \max\left\{ 0, \lambda\submax(Z,L) \right\}.
   \]
   Hence, the essential atoms are given by the maximal generalized
   eigenvectors and their convex combinations:
   \[
     \Escr\As(Z) = \conv\set{uu^T | \ip{u}{Z u} = \lambda\submax(Z,L)\cdot \ip{u}{Lu}}.
   \]
 \end{example}

\section{Alignment as optimality} \label{sec-manifestations}

A pair of vectors $(x,z)$ that is aligned with respect to an atomic set inform
each other about their respective supports. If the two vectors are related
through a gradient map of a convex function, then the alignment condition can be
interpreted as an optimality condition for a constrained or regularized
optimization problem. The alignment condition can also be interpreted as
providing an optimality certificate for the problem of finding minimum gauge
elements of a convex set. This section describes both perspectives.

\subsection{Regularized smooth problems}

Consider the three related convex optimization problems
\begin{subequations} \label{eq:smoothopto}
\begin{alignat}{4}
    \label{eq:smoothopto-unconstrained}
  &\minimize{x} &\quad& f(x) \mathrlap{{} + \rho\gamma\Cs(x),}
  \\\label{eq:smoothopto-set-constrained}
  &\minimize{x} &\quad& f(x) &\ &\st\ & \gamma\Cs(x)&\leq \alpha,
  \\\label{eq:smoothopto-level-constrained}
  &\minimize{x} &\quad& \gamma\Cs(x) && \st\ & f(x) &\leq \tau,
\end{alignat}
\end{subequations}
where $\rho$, $\alpha$, and $\tau$ are positive parameters. Note that the
constraint $\gauge\Cs(x)\le\alpha$ is equivalent to the constraint
that $x$ is in the set $\alpha\Cscr$. 
\Cref{blanket-assumption} on $\Cscr$ continues to hold throughout.
\begin{theorem}[Optimality] \label{prop-optimality-smooth} Let $f:\Re^n\to\Re$
  be a differentiable convex function and $\Cscr\subset\Re^n$. Assume
  that~\eqref{eq:smoothopto-level-constrained} is strictly feasible. For each of
  the problems in~\eqref{eq:smoothopto}, a feasible point $x^*$ is optimal if
  and only it is $\Cscr$-aligned with $z^*:=-\nabla f(x^*)$.
  \end{theorem}
\begin{proof}
  First consider the unconstrained
  problem~\eqref{eq:smoothopto-unconstrained}. A vector $x^*$ is a
  solution if and only if
  \[0\in\nabla f(x^*) + \rho\partial\gauge\Cs(x^*).\] Equivalently,
  \[
    \rho^{-1}z^* \in \partial\gauge\Cs(x^*)=\partial\sigma\Csp(\xstar)
    \equiv\Face\Csp(\xstar).
  \]
  Then by \cref{cor-general-alignment}, this condition is equivalent to the
  $\Cscr$-alignment of the pair $(x^*,z^*)$.

  Next, consider the gauge constrained
  problem~\eqref{eq:smoothopto-set-constrained}. Because
  $\gamma\Cs(x)\leq \alpha$ is equivalent to $\alpha^{-1}x\in\Cscr$,
  a feasible vector $x^*$ is optimal if and only if
  \[0\in\nabla f(x^*) + \partial\delta\Cs(\alpha^{-1}x^*) \text{i.e.,} -\nabla
    f(x^*) \in \partial\delta\Cs(\alpha^{-1}x^*),\] where $\delta\Cs$
  is the indicator function for set $\Cscr$. By
  \cite[Theorem~23.5]{rockafellar1970convex} it follows that
  $x^* \in \alpha\partial\sigma\Cs(z^*)$, and thus by
  \cref{cor-general-alignment}, the pair $(x^*,z^*)$ is
  $\Cscr$-aligned.

  Finally, consider the level constrained
  problem~\eqref{eq:smoothopto-level-constrained}. Let
  $\Pscr = \set{x | f(x) \leq \tau}$. The hypothesis on $f$ ensures that $\Pscr$ has a non-empty relative interior. Then a feasible vector $x^*$ is
  optimal if and only if
  \[
    0 \in \partial(\gamma\Cs(x^*) + \delta_\Pscr(x^*))
    = \partial\gamma\Cs(x^*) + \partial\delta_\Pscr(x^*),
  \]
  where the equality follows from
  Rockafellar~\cite[Theorem~23.8]{rockafellar1970convex}.

  Now we consider two cases. If $f(x^*) < \tau$, then it follows that $0 \in
  \partial\gamma\Cs(x^*)$, and thus $x^* \in \rec\Cscr$. Then by
  \cref{cor-general-alignment}, the pair $(x^*,z^*)$ is $\Cscr$-aligned. If
  $f(x^*) = \tau$, then by \cite[Theorem~1.3.5]{hiriart-urruty01}, there exists
  a positive scalar $\lambda$ such that
  \[0 \in \partial\gamma\Cs(x^*) + \lambda\nabla f(x^*) \text{i.e.,} z^* \in
    \cone \partial\gamma\Cs(x^*).\] Then by \cref{cor-general-alignment}, the
  pair $(x^*,z^*)$ is $\Cscr$-aligned.
\end{proof}

\subsubsection{Objective value bound}
With only slightly more effort, \Cref{prop-optimality-smooth} implies that the residual
\[
  g\Cs(x) = \alpha^*\sigma\Cs(z_x) - \ip{x}{z_x}, 
  \text{with} z_x := -\nabla f(x),
\]
of the polar inequality, where $\alpha^*$ is an upper bound on the gauge value
$\gamma\Cs(x^*)$ of any optimal solution, bounds the difference between the
objective value of $f(x)$ and the optimal value $f(x^*)$. Generally a bound
$\alpha^*$ is not available. The notable exception, however, is for problems of
the form~\eqref{eq:smoothopto-set-constrained}, where feasibility implies that
$\gauge\Cs(x^*)\le\alpha$, and in that case we may simply take
$\alpha^*=\alpha$. To see how $g_c$ provides the bound on the optimal value of
$f$, note that
\begin{align*}
  f(\xstar) &\ge f(x) + \ip{\xstar-x}{\nabla f(x)}
          \\&\ge f(x) + \min_{a\in\alpha^*\Cscr}\, \ip{a-x}{\nabla f(x)} 
          \\&= f(x) + \ip{x}{z} - \alpha^*\sigma\Cs(z),
\end{align*}
where the first inequality follows from the subgradient inequality.
Rearranging terms and using the definition of $g_c$, we obtain the bound
\[
 g_c(x) \ge f(x)-f(\xstar) \quad\forall x. 
\]
A similar bound is derived by Jaggi~\cite{jaggi2013revisiting} in the context of the
conditional gradient method applied to~\eqref{eq:smoothopto-set-constrained} and
by Ndiaye et al.~\cite{ndiaye2016gap}. 

\subsubsection{Conditional gradient and atomic alignment}

Conditional gradient (CG) methods
\cite{jaggi2013revisiting,frank1956algorithm,dunn1978conditional} naturally
exhibit the atomic alignment property in several ways. Here we describe one
property related to alignment that can be used to develop computationally
efficient variations for this class of methods.

In its simplest form, the CG method applies to problems such
as~\eqref{eq:smoothopto-set-constrained}. Because here we wish to make the
atomic set explicit, we express that problem as
\begin{equation} \label{eq:CG-generic-problem}
   \minimize{x\in\Ascr}\enspace f(x).
\end{equation}
We adopt the simplifying assumption that $\Ascr$ is compact so that every
direction exposes a face. The iterates of the basic CG method are summarized in \cref{algo:conditional-gradient-vanilla}. 
\begin{algorithm}[t]
  \DontPrintSemicolon
  \KwIn{$x^{(0)}\in\Ascr$,\ $\epsilon>0$} 
  \For{$k=0,1,2,\ldots$}{
    $z\k=-\nabla f(x\k)$\;
    $a\k\in\tau\FaceA(z\k)$\label{algo:cg-expose}\;%
    \lIf{$\ip{a\k-x\k}{z\k}<\epsilon$}{stop}
    $x\kp1=\theta\k a\k + (1-\theta\k)x\k$, \quad $\theta\k\in(0,1)$\label{algo:cg-merge}\;
  }
  \Return{$x\k$}
  \caption{Conditional gradient method for~\eqref{eq:CG-generic-problem}\label{algo:conditional-gradient-vanilla}}
\end{algorithm}
The linear minimization oracle (LMO) in Step~\ref{algo:cg-expose} selects an atom or
a convex combination of atoms from the set $\Ascr$ exposed by the current
negative gradient $z\k\equiv-\nabla f(x\k)$. In the language of atomic alignment, the LMO step
selects an atom $a\k$ that is aligned with $z\k$. In particular, observe that
\[
  \ip{a\k}{z\k} = \sigma\Cs(z\k) = \gauge\Cs(a\k)\cdot\sigma\Cs(z\k),
\]
where the second equality follows because $a\k\in\Cscr$, and so
$\gauge\Cs(a\k)\le1$.
 
Step~\ref{algo:cg-merge} merges the selected element
$a\k$ with the collection of atoms that have been exposed through iteration $k$,
and which are represented as an aggregate in the iterate $x\k$. Various choices
for the steplength $\theta\k$ exist, including linesearch, which requires
additional evaluations of the function $f$ to ensure sufficient decrease, and a
decaying steplength that follows a predetermined schedule.

The recent appeal of these methods lies with the computational efficiency of the
linear minimization oracle for many important special cases, especially cases
where projections or proximal operations are not computationally feasible. The
unit nuclear-norm ball described in \cref{example-nuclear-norm}
illustrates the point: projection of an $n$-by-$m$ matrix $X$ onto the set
$\Ascr_*:=\set{Z|\|Z\|_*\le1}$ is the matrix
\begin{equation*}
  \proj_{\Ascr_*}(X) = U\bar\Sigma V^T \text{with}
  \bar\Sigma = \Diag(\min\{1,\sigma_i(X)\}^{\min\{m,n\}}_{i=1}).
\end{equation*}
where $X=U\Diag(\{\sigma_i(X)\}_{i=1}^{\min\{m,n\}})V^T$ is the singular-value decomposition of $X$.
Thus, the
projection operation requires computing all singular triples of $X$ larger than 1.
In contrast, the linear minimization oracle in Step~\ref{algo:cg-expose}
requires only computing one of the maximal singular triples of the negative
gradient (a matrix, in this case). For this reason, the CG method often features
in applications of matrix completion \cite{jaggi2010simple,shalev2011large,lee2009efficient}.

We express the merge step at iteration $k$ recursively as
\begin{equation}
  \label{eq:25}  x\k = \sum_{i=1}^k\widehat\theta^{(i)}a^{(i)},
  \quad
  \widehat\theta^{(i)}:=\theta^{(i)}\prod_{j=1}^i(1-\theta^{(j)}).
\end{equation}
This expression makes explicit the one-atom-at-a-time construction of the
current iterate $x\k$, each taken from a face exposed by the negative gradients.
Thus,
\[
  x\k \in \sum_{i=1}^k\widehat\theta^{(i)}\FaceA(z^{(i)}).
\]
In an idealized, perfectly greedy run of the algorithm, the sequence of exposed
faces $\FaceA(z\k)$ are expanding, i.e., $\FaceA(z\k)\subseteq\FaceA(z\kp1)$,
and converge to an optimal face $\FaceA(z^*)$, where $z^*:=-\nabla f(x^*)$. But
in general, we do not expect such efficiency, and may inadvertently collect many
sets of atoms that are not at all related to the optimal face, so that some
atoms $a\k\notin\suppa(x^*)$. Thus, the computed decomposition~\eqref{eq:25} at any
iteration $k$ may contain atoms not in the optimal support $\suppa(x^*)$. In
applications such as matrix-completion, described in
\cref{example:dual-conditional-gradient} below, the cost of storing
intermediate atoms $a\k$---say, as singular pairs $(u\k,v\k)$---can be
prohibitively expensive for large problems. Various modifications of the basic
CG method aim to compress or trim the collected atoms to alleviate unnecessary storage~\cite{rao2015forward}.

In the case of a least-squares objective function, the alignment principle
provides a simple device that short-circuits the need for storing
intermediate atoms, as illustrated in the following example.

\begin{example}[Delayed atom generation]\label{example:dual-conditional-gradient}
  Consider the low-rank matrix completion problem
  \begin{equation*}
    \minimize{X\in\Real^{m\times n}}\enspace\half\|\Omega\circ(X - B) \|_F^2
    \enspace\st\enspace\|X\|_* \le\tau.
  \end{equation*}
  This problem appears in recommender systems~\cite{bell2007lessons}, where the
  $(i,j)$th element of the sparse matrix $B$ records the ratings score given by user
  $i$ for product $j$. Ratings are observed only for a subset of user-product
  pairs indexed by the binary mask
  \[
    \Omega_{ij} =
    \begin{cases}
      1 & \mbox{if user $i$ has rated product $j$;}
    \\0 & \mbox{otherwise.}
    \end{cases}
  \]
  The goal is to predict the unseen ratings, captured in the dense unknown
  matrix $X$. A structural low-rank assumption is used to capture an
  ``archetype'' phenomenon---users who often like the same movies serve as good
  predictors for one another, and movies that are liked by the same users probably
  are also similar. Therefore, we can consider each user as a sparse linear
  combination of archetypal individuals (and similarly with products), where the
  inner product of their feature vectors give the same prediction rating. The
  nuclear-norm constraint on $X$ is a common approach for encouraging low-rank
  solutions~\cite{recht2010guaranteed}.

  Most of the computational cost of \Cref{algo:conditional-gradient-vanilla}
  applied to this problem is represented in Step~\ref{algo:cg-expose}, which requires
  calculating a maximal singular pair of the current negative gradient
  $Z\k:=-\nabla f(X\k) = \Omega\circ(B-X\k)$. This is a sparse matrix indexed
  by $\Omega$.
  (\Cref{algo:conditional-gradient-vanilla} is written with lower-case symbols
  to denote vectors, but we use upper-case symbols here to denote the matrix
  iterates for this problem.) Thus, the atoms $a\k$ are outer products of the
  unit-norm vector pairs $(u\k,v\k)$ that satisfy $\ip{u}{Z\k v}=\tau\sigma\submax(Z\k)$.
  The key limitation of this approach is that either the atoms are aggregated
  into a dense iteration matrix $X\k$, or are stored as a sequence of pairs
  $\{(u^{(i)},v^{(i)})\}_{i=1}^k$. In either case, the memory requirements are
  prohibitive for anything but small problems.

  \begin{algorithm}[t]
    \DontPrintSemicolon
    \SetKwComment{tcp}{[}{]}
    \KwIn{$\Omega$, $B$, $\ell$}
    $R\k=\Omega\circ B$;\enspace $Q\k=0$\;
    \For{$k=1,2,\ldots$}{
      $Z\k = \Omega\circ R\k$\;
      $(u,v) = \mbox{\tt svds($Z\k$, 1)}$\tcp*{expose atom $A\k\equiv \tau uv^T$}
      $\Delta R\k = \Omega\circ(\tau uv^T)-Q\k$ \tcp*{$\Omega\circ(\tau uv^T)=\tau(u_i v_i)_{(ij)\in\Omega}$}
      \label{cg-dual-optimality} \lIf{$\ip{\Delta R\k}{R\k}<\epsilon$}{stop}
      $\theta\k = \min\{1,\, \ip{\Delta R\k}{R\k}/\|\Delta R\k\|^2\}$ \tcp*{exact linesearch on quadratic}
      $R\kp1 = R\k - \theta\k\Delta R\k$\;
      $Q\kp1 = Q\k + \theta\k\Delta R\k$\;
    }
    $(U_\ell, V_\ell, \Sigma_\ell) = \mbox{\tt svds($Z\k$, $\ell$)}$\tcp*{top
      $\ell$ singular vectors} \label{cg-dual-expose} 
    $S_\ell\in \argmin\left\{\half\|\Omega\circ(U_\ell SV_\ell^T)-b)\|_2^2
      \mid \trace(S)\le\tau,\ S\succeq0\right\}$\label{cg-dual-recover}\;
    \Return{$(U_\ell,S_\ell,V_\ell)$}
    \caption{Dual conditional gradient for~\eqref{eq:CG-generic-problem}
      with least-squares objective. The linear operator $\Omega$\label{algo:dual-conditional-gradient}.}
  \end{algorithm}
  \Cref{algo:dual-conditional-gradient} describes a ``dual'' version of
  the CG method shown in Algorithm~\ref{algo:conditional-gradient-vanilla},
  similar to the approach used by Yursever et al.~\cite{yurtsever2017sketchy},
  who maintain a low-memory random sketch of the primal iterate. Instead, here
  we completely forgo any reference of the primal iterate during the CG
  iterations, and only update a sequence of dual vectors $Z\k\to Z^*\equiv
  -\nabla f(X^*)$. The corresponding primal solution $X^*$ is subsequently
  recovered using the alignment between $X^*$ and $Z^*$, as spelled out
  by \cref{prop-optimality-smooth}. This technique was used by Friedlander and
  Mac\^edo~\cite{friedlander2016low} for low-rank SDPs.

  The sequence of iterates $Z\k$, $R\k$, and $Q\k$ respectively, coincide with
  the negative gradient $-\nabla f(X\k)$, residual $R\k=\Omega\circ(X\k-B)$, and
  restricted primal iterate $Q\k=\Omega\circ X\k$, where the feasible primal
  iterate $X\k$ is held implicitly. At the termination of the loop,
  Steps~\ref{cg-dual-expose} and~\ref{cg-dual-recover} use the fact that the
  latest iterate $Z\k$ exposes the range and column spaces of the solution.
  Thus, it can be used recover a rank-$\ell$ matrix that best approximates (in
  the sense of the Frobenious norm) the corresponding primal iterate $X\k$; see
  \Cref{example-nuclear-norm}. Step~\ref{cg-dual-optimality} is equivalent to
  the optimality test in \Cref{algo:conditional-gradient-vanilla}
  because \[\ip{\Delta R\k}{R\k} = \ip{\Omega\circ(\tau uv^T +
      X\k)}{\,R\k}=\ip{A\k-X\k}{Z\k},\] where $A\k:=\tau uv^T$. The linesearch
  parameter $\theta\k$ is an exact minimizer of $\|R\k-\theta\Delta R\k\|_F$
  over $\theta$.

  \Cref{table-cg-on-matrix-completion} lists the results of applying the primal
  and dual CG variants on a random matrix-completion problem. For varying
  problem sizes with $m=n$, we generate the binary mask $\Omega$ with 10\%
  nonzeros, and the observation matrix is generated
  \[
    B = \Omega\circ(UV^T+0.1\cdot N), 
  \]
  where $U\in\Real^{m\times r}$, $V\in\Real^{n\times r}$, and $N$ are generated
  i.i.d.~from a normal Gaussian distribution. The ``true rank'' $r=\mathop{\bf
    round}(m/100)$. Interestingly, the multiplicity of the computed dual
  solution $Z^*$ was always $1$, which made the primal-recovery phase
  (Step~\ref{cg-dual-recover} of~\Cref{example:dual-conditional-gradient})
  trivial. The residual values between the two variants are the same, confirming
  that they recover solutions of similar quality, but the dual variant is
  significantly faster because it does not need to manipulate storage for the dual 
  iterate $X\k$.
  
  \begin{table}[t] 
  \centering \pgfplotstabletypeset[%
  columns={size,objpr,rankpr,timepr,objdu,rankdu,timedu},
  columns/size/.style={
    column name={$m=n$},column type={r} },
  columns/objpr/.style={
    column
    name={residual},precision=1,column type=r }, columns/objdu/.style={ column
    name={residual},precision=1,column type=r }, columns/rankpr/.style={ column
    name={rank},postproc cell content/.append style={ /pgfplots/table/@cell
      content/.add={}{}} }, columns/rankdu/.style={ column name={rank},postproc
    cell content/.append style={ /pgfplots/table/@cell content/.add={}{}} },
  columns/timepr/.style={ column name={time},fixed,fixed
    zerofill,precision=1,column type=r }, columns/timedu/.style={ column
    name={time},fixed,fixed zerofill,precision=1,column type=r }, every head
  row/.style={ before row={ \toprule
          \multicolumn{1}{c}{size}
         &\multicolumn{3}{c}{primal CG}
         &\multicolumn{3}{c}{dual CG}
        \\\cmidrule(lr){1-1}\cmidrule(lr){2-4}\cmidrule(lr){5-7}
      },
      after row=\midrule
    },
    every last row/.style={after row=\bottomrule},
    ]{
			size  	 objpr  rankpr  objdu rankdu  timepr   timedu
			100		   10.31  6			10.28 1     0.039    0.039
			250		   25.17  6 		25.16 1     0.074    0.075
			1000	  100.39  6		 100.39 1     1.31     0.32
			5000    501.69  6    501.69 1    48.34    11.38
			10000	  998.86  6  	 998.86 1   242.9     63.32
    }
    \caption{Performance of the primal and dual variants of conditional gradient
      for the matrix-completion problem
      (\Cref{example:dual-conditional-gradient})
      after 10 iterations of
      Algorithm~\ref{algo:conditional-gradient-vanilla} (primal CG) and
      Algorithm~\ref{algo:cg-expose} (dual CG). Estimated rank of final solution
      is computed as the smallest number of singular values that account for 90\% of its
      Frobenious norm. Time is measured in seconds.\label{table-cg-on-matrix-completion}}
  \end{table}
\end{example}

\subsection{Alignment in gauge optimization}\label{sec-gauge-optimization}

The alignment property characterizes the optimality for the polar dual
pair
\begin{equation}   \label{eq:gaugeduality}
  \begin{array}{l@{\enspace}l}
    \minimize{x} & \gauge\Cs(x)\\ 
    \st & x\in \Dscr,
  \end{array}
  \qquad
  \begin{array}{l@{\enspace}l}
    \minimize{z} & \sigma\Cs(z)\\
    \st & z\in \Dscr',
  \end{array}
\end{equation}
where $\Dscr\subset\Re^n\backslash\{0\}$ is any closed convex set and
$\Dscr':=\set{z|\ip x z \ge 1\ \forall x\in\Dscr}$ is its antipolar. This class
of problems and its applications is described in detail by Freund
\cite{freund1987dual} and by Friedlander, Mac\^edo, and Pong
\cite{friedlander2014gauge}.

\begin{proposition}[Polar duality] A pair of primal-dual feasible vectors
  $(x,z)\in\Dscr\times\Dscr'$ is primal-dual optimal
  for~\eqref{eq:gaugeduality} if and only if they are
  $\Cscr$-aligned and $\ip x z = 1$.
\end{proposition}

\begin{proof}
  First, assume that the pair $(x,z)$ is primal-dual optimal
  for~\eqref{eq:gaugeduality}, then by the strong
  duality~\cite[Corollary~5.2]{friedlander2014gauge}, we have
  \begin{equation} \label{eq:6}
    1 = \ip x z = \gauge\Cs(x)\cdot\sigma\Cs(z).
  \end{equation}
  We prove the other direction by contradiction. Assume $(x,z)$ are
  $\Cscr$-aligned and $\ip x z = 1$ and suppose there exists
  $\xhat \in C$ such that $\gauge\Cs(\xhat) < \gauge\Cs(x)$. Then it
  follows that
  \[\ip{\xhat}{z} \overset{(a)}{\geq} 1 \overset{(b)}{\vphantom{\ge}=} \gauge\Cs(x)\cdot\sigma\Cs(z) > \gauge\Cs(\xhat )\cdot\sigma\Cs(z),\]
  where the inequality (a) follows from the definition of the
  antipolar $\Dscr'$, and the equality (b) follows
  from~\eqref{eq:6}. This violates the polar gauge inequality, and
  thus leads to a contradiction.
\end{proof}

\section{Alignment in convolution of atomic sets}

The notions of atomic decomposition and alignment that we have
discussed thus far are all tied to a single atomic set
$\Ascr$. Correspondingly, the regularized optimization problems
considered in~\cref{sec-manifestations} involve only a single
regularization function $\gauge\As$ meant to encourage minimizers 
sparse with respect to $\Ascr$.  Richer atomic decompositions and
regularized formulations, however, may be obtained by combining
different atomic sets. We describe in this section approaches that
combine multiple atomic sets $\Ascr_1$ and $\Ascr_2$. Informally, we
seek to decompose a vector $x$ as the additive decomposition of the
form
\begin{equation} \label{eq:24}
  x = x_1 + x_2 \quad \mbox{where each $x_i$ is $\Ascr_i$ sparse}.
\end{equation}
This operation appears often in models for separating signals, also
known as
demixing~\cite{dtds06,wright2009robust,candes2011robust,wright2013compressive,McCoy2014,oymak2017universality}.

A common approach is to directly construct an aggregate atomic set
as the union of simpler sets $\Ascr_1$ and $\Ascr_2$, each with a
special structure that reflects a useful decomposition. The union of
atomic sets, in fact, corresponds to the infimal sum convolution of
their corresponding gauge functions, as we show below.

Our main focus, however, is an alternative and less-often used
approach that forms the aggregate atomic set as the Minkowski sum
\begin{equation*} \label{eq:23}
  \Ascr_1+\Ascr_2 = \set{a_1+a_2|a_1\in\Ascr_1,\ a_2\in\Ascr_2}
\end{equation*}
of the simpler atomic sets, which directly mirrors the desired
decomposition in~\eqref{eq:24}. As with the union operation, the sum
of atomic sets also corresponds to a convolution operation of the
corresponding gauge functions, except that in this case it is
\emph{polar convolution} \cite{FriedlanderMacedoPong:2019}, rather
than sum convolution.

The sum of atomic sets and the connection to polar convolution
allows us to deduce properties of alignment for the
constituent sets, and thus to suggest dual approaches similar to
\cref{algo:dual-conditional-gradient} for solving the
optimization formulations that arise in demixing applications.

Below we only consider two distinct atomic sets $\Ascr_1$ and
$\Ascr_2$. The convolution of three or more sets is an obvious
extension. Recall our notational convention that for any set $\Dscr$,
$\gauge_{\scriptscriptstyle\Dscr} \equiv
\gauge_{\scriptscriptstyle\conv\Dscr}$.

\subsection{Atomic sum and polar convolution} \label{sec:sum-max-conv}

One important application of the alignment principles that we discuss in this section is in the analysis of the various demixing problems
\begin{subequations} \label{eq:demixing-problems}
\begin{alignat}{4}
    \label{eq:demixing-unconstrained}
    &\minimize{x_1,\,x_2} &\quad& f(x_1+x_2) \mathrlap{{}
      + \rho\max\set{\gamma\Aso(x_1),\,\gamma\Ast(x_2)}}
  \\\label{eq:demixing-gauge-constrained}
  &\minimize{x_1,\,x_2} &\quad& f(x_1+x_2)
  &\ &\st\ & \max\set{\gamma\Aso(x_1),\,\gamma\Ast(x_2)}&\leq \alpha,
  \\\label{eq:demixing-level-constrained}
  &\minimize{x_1,\,x_2} &\quad& \max\set{\gamma\Aso(x_1),\,\gamma\Ast(x_2)} && \st\ & f(x_1+x_2) &\leq \tau.
\end{alignat}
\end{subequations}
These three problems are in fact just special cases of the regularized
formulations in~\eqref{eq:smoothopto}, where $\gauge_\Ascr(x)$ is
replaced by the function
$\max\set{\gauge\Aso(x_1),\,\gauge\Ast(x_2)}$. One of the main goals of this section is to prove the following corollary to \cref{prop-optimality-smooth}.
\begin{corollary}[Optimality and atomic
  sums] \label{cor:optimality-sums} Let $f:\Real^n\to\Real$ be a
  differentiable convex function and $\Ascr_i\subset\Real^n$ for
  $i=1,2$. Assume that at least one set $\Ascr_i$ contains the origin
  in its interior, and that the
  problem~\eqref{eq:demixing-level-constrained} is strictly
  feasible. For each of the problems in~\eqref{eq:demixing-problems},
  a feasible pair $(x_1^*,x_2^*)$ is optimal if and only if $x_i^*$ is
  $\Ascr_i$-aligned with $z^*:=-\nabla f(x_1^*+x_2^*)$.
\end{corollary}

Before we can establish the proof of this result, however, we first
establish the close relationship between the sum of atomic
sets and polar convolution. This connection is an important analytical
tool in its own right.

\subsubsection{Polar convolution}

The polar convolution of two gauges $\gauge\Aso$ and $\gauge\Ast$ results in the function
\begin{align}
  (\gauge\Aso\maxconv\gauge\Ast)(x)
  &= \inf_w\max\set{\gauge\Aso(w),\ \gauge\Ast(x-w)}.
    \label{eq:max-conv}
\end{align}
This operation first appears in
Rockafellar\cite[Theorem~5.8]{rockafellar1970convex} for general
convex functions, and is subsequently analyzed by Seeger and
Volle\cite{SeegerVolle95}. When specialized to gauge functions, as
shown in~\eqref{eq:max-conv}, this convolution operation is tightly
connected to the polarity operation on the defining atomic sets. In
that case, Friedlander et al.~\cite{FriedlanderMacedoPong:2019} refer
to the operation as \emph{polar convolution}.

Polar convolution is in fact the functional counterpart to set addition.
\begin{proposition}[Polar convolution of
  gauges] \label{prop:max-convolution} Let $\Ascr_1$ and $\Ascr_2$ be
  non-empty closed convex sets that contain the origin. If at least
  one set contains the origin in its interior, then the polar
  convolution of the gauges $\gauge\Aso$ and $\gauge\Ast$ is the gauge
  \[
    \gauge\Aso\maxconv\gauge\Ast = \gauge_{\scriptscriptstyle\Ascr_1+\Ascr_2}.
  \]
\end{proposition}
\begin{proof}
  The hypothesis that one of the sets $\Ascr_1$ and $\Ascr_2$ contains
  the origin implies that the corresponding gauge (say, $\gauge\Aso$) is finite and
  therefore continuous. Thus, 
  \begin{align*}
    \gauge\Aso\maxconv\gauge\Ast
     = (\gauge_{\scriptscriptstyle\Ascr_1\polar}
      + \gauge_{\scriptscriptstyle\Ascr_1\polar})\polar
     = \gauge_{\scriptscriptstyle\Ascr_1+\Ascr_2},
  \end{align*}
  where the first equality follows from
  \cite[Lemma~3.3]{FriedlanderMacedoPong:2019} and the continuity of
  $\gauge\Aso$, and the second equality follows from
  \cite[Lemma~3.4]{FriedlanderMacedoPong:2019}.
\end{proof}

For the remainder of this section, we assume that one of the gauges
$\gauge_{\scriptscriptstyle\Ascr_i}$ is continuous, which holds if the
origin is contained in the interior of $\Ascr_i$.

\subsubsection{Alignment to the sum of sets}

The polar convolution operation, which mixes atoms via the sum of
sets, has the appealing property that it explicitly decomposes a
vector as a sum of elements, each belonging to one of the atomic
sets. In particular, evaluating the polar convolution
\begin{equation} \label{eq:22}
  (\gauge\Aso\maxconv\gauge\Ast)(x)
  = \gauge_{\scriptscriptstyle\Ascr_1+\Ascr_2}(x)
  = \inf_{x_1,\,x_2}\max\set{\gauge\Aso(x_1),\,\gauge\Ast(x_2)|x=x_1+x_2}
\end{equation}
at a point $x$ implicitly generates a decomposition
\[
  x = \sum_{\mathclap{a\in\Ascr_1+\Ascr_2}}c_aa
  = \sum_{\mathclap{\genfrac{}{}{0pt}{}{a_1\in\Ascr_1}{a_2\in\Ascr_2}}}c_a(a_1+a_2)
    = x_1 + x_2,
\]
where each $x_i\in\cone\Ascr_i$. Moreover, it is a straightforward
consequence of optimality for~\eqref{eq:22} that
$\gauge\subAsum(x)=\gauge\Aso(x_1)=\gauge\Ast(x_2)$.

\begin{theorem}[Alignment in polar convolution] \label{th:polar-alignment}
  Suppose that the pair $(x,z)$ is $(\Ascr_1+\Ascr_2)$-aligned and
  \[
    \gauge_{\scriptscriptstyle\Ascr_1+\Ascr_2}(x)
    = \gauge\Aso(x_1) = \gauge\Ast(x_2),
    \text{where} x=x_1+x_2.
  \]
  Then the pair $(x_i,z)$ is $\Ascr_i$-aligned for $i=1,2$.
\end{theorem}
\begin{proof}
   Because $x$ and $z$ are $(\Ascr_1+\Ascr_2)$-aligned,
  \[
    \gauge\subAsum(x)\cdot\sigma\subAsum(z)
    = \ip{x}{z} = \ip{x_1}{z} + \ip{x_2}{z}.
  \]
  Use the fact that $\sigma\subAsum=\sigma\Aso+\sigma\Ast$ and rearrange terms to deduce that
  \[
    \sigma\Aso(z)+\sigma\Ast(z) =
      \ip*{\frac{x_1}{\gauge\Aso(x_1)}}{z}
    + \ip*{\frac{x_2}{\gauge\Aso(x_2)}}{z}.
  \]
  But because $x_1\in\Ascr_1$ and $x_2\in\Ascr_2$, it follows that
  \[
    \gauge\Aso(x_1)\cdot\sigma\Aso(z) = \ip{x_1}{z}
    \text{and}
    \gauge\Aso(x_2)\cdot\sigma\Aso(z) = \ip{x_2}{z},
  \]
  which establish, respectively, that each $(x_i,z)$ is $\Ascr_i$-aligned.
\end{proof}
\begin{figure}[t]\label{figure-convolution}
  \centering
  \includegraphics[page=9]{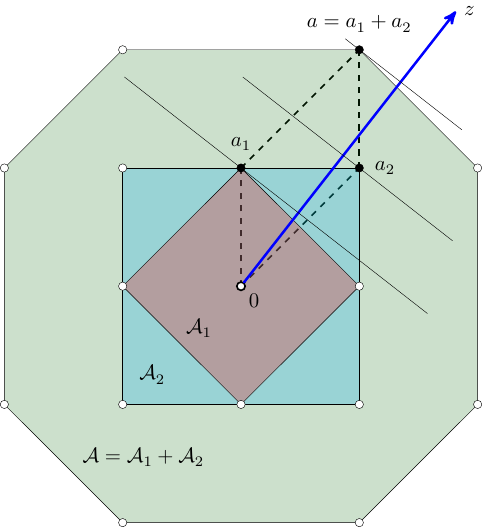}
  \caption{Illustration of the polar alignment principle for atomic sums, as described by
    \Cref{th:polar-alignment} (alignment in polar convolution). The vector $z$ simultaneously exposes atoms,
    indicated by black dots, in
    the atomic sets $\Ascr_1$ and $\Ascr_2$, and also in the sum of atomic sets $\Ascr=\Ascr_1+\Ascr_2$.}
\end{figure}
\subsubsection{Proof of \cref{cor:optimality-sums}}

The first step in the proof is to establish that the regularized optimization problems in~\eqref{eq:demixing-problems} are equivalent, respectively, with the problems
\begin{subequations} \label{eq:demixing-problems-conv}
\begin{alignat}{4}
    \label{eq:demixing-unconstrained-conv}
    &\minimize{x} &\quad& f(x) \mathrlap{{}
      + \rho\gamma\subAsum(x)}
  \\\label{eq:demixing-gauge-constrained-conv}
  &\minimize{x} &\quad& f(x)
  &\ &\st\ & \gamma\subAsum(x)&\leq \alpha,
  \\\label{eq:demixing-level-constrained-conv}
  &\minimize{x} &\quad& \gamma\subAsum(x) && \st\ & f(x) &\leq \tau.
\end{alignat}
\end{subequations}
We establish the equivalence for~\eqref{eq:demixing-unconstrained-conv}; the
equivalence for~\eqref{eq:demixing-gauge-constrained-conv}
and~\eqref{eq:demixing-level-constrained-conv} follows the same line of
reasoning. Observe that
\begin{align*}
  &\inf_{x_1,\,x_2} \Bigl\{ f(x_1+x_2)
    + \rho\max\set{\gamma\Aso(x_1),\,\gamma\Ast(x_2)} \Bigr\}
  \\= &\inf_{\phantom{x_1,\, x_2}\mathllap{x,\, x_1}} \Bigl\{ f(x) + \rho\max\set{\gamma\Aso(x_1),\, \gamma\Ast(x - x_1)} \Bigr\}
  \\= &\inf_{\phantom{x_1,\, x_2}\mathllap{x\ \ }} \Bigl\{ f(x) + \rho\inf_{x_1}\max\set{\gamma\Aso(x_1),\, \gamma\Ast(x - x_1)} \Bigr\}
  \\= &\inf_{\phantom{x_1,\, x_2}\mathllap{x\ \ }} \Bigl\{ f(x) + \rho\gamma\subAsum(x) \Bigr\},
\end{align*}
where the last equality follows from the definition of polar
convolution~\eqref{eq:max-conv} and \cref{prop:max-convolution}.

Next, we use \cref{prop-optimality-smooth} to establish that a point
$\xstar$ is a solution to one of the three
problems~\eqref{eq:demixing-problems-conv} if and only if $x^*$ is
$(\Ascr_1+\Ascr_2)$-aligned with $z^*:=-\nabla f(x^*)$. The
equivalence of the formulations~\eqref{eq:demixing-problems-conv}
and~\eqref{eq:demixing-problems} means that $x^*=x_1^*+x_2^*$, where
$x^*_1$ and $x^*_2$ are optimal
for~\eqref{eq:demixing-problems}. Moreover, optimality of $x^*_1$ and
$x^*_2$ implies that
$\gauge_{\scriptscriptstyle\Ascr_1+\Ascr_2}(x^*) = \gauge\Aso(x^*_1) =
\gauge\Ast(x^*_2)$. Thus, \cref{th:polar-alignment} applies in this case and each
pair $(x^*_i,z^*)$ is $\Ascr_i$-aligned.\qed

\subsubsection{Morphological component analysis} \label{sec-morphological-component-analysis}

We show how the alignment principle can be used as part of a demixing
application in signal separation, also known as morphological component
analysis~\cite{dtds06}. Our discussion below focuses on demixing using the
constrained formulation~\eqref{eq:demixing-gauge-constrained-conv}, but can be
easily extended to the other two formulations.

Suppose that $x^*$ is the solution of
\eqref{eq:demixing-gauge-constrained-conv}. Then we know $x^*=x_1+x_2$ for some
$x_i\in\alpha\Ascr_i$. We recover the constituent components $x_i$ using two
stages. In the first stage, we apply the conditional gradient method
(\Cref{algo:conditional-gradient-vanilla})
to~\eqref{eq:demixing-gauge-constrained-conv} with $\Ascr:=\Ascr_1+\Ascr_2$ to
obtain the negative gradient $z^*=-\nabla f(x^*)$. (The primal iterate $x\k$
does not need to be stored.) The key to the efficient application of this method
is to recognize that the exposed face of the sum of sets is equal to the sum of
exposed faces, i.e.,
\[
  \Face{\!\scriptscriptstyle\Ascr_1+\Ascr_2}(z)=\Face\Aso(z)+\Face\Ast(z).
\]
Thus, Step~\ref{algo:cg-expose} in the CG method can be implemented using
separate procedures available for exposing a face in each of the atomic sets
$\Ascr_i$.

In the second stage, we use $z^*$ to expose the atoms in each component $x_i$.
\Cref{th:polar-alignment} asserts that each $x_i$ is $\Ascr_i$-aligned
with the negative gradient $z^*:=-\nabla f(x^*)$, and therefore exposes the atoms in
$\Ascr_i$ that supports $x_i$. Thus, each component $x_i$ can be recovered as
the solution of the reduced optimization problem
\[
  \minimize{x_1,\,x_2}\enspace f(x_1+x_2)
  \enspace\st\enspace
  \gauge_{\scriptscriptstyle\Escr_{\Ascr_i}(z^*)}(x_i)\le\alpha,\ i=1,2.
\]
The underlying assumption, of course, is that the exposed face $\Face\Asi(z^*)$
containing the relevant atoms has small dimension, since otherwise this problem
could be as expensive as the original problem. A variety of algorithms can be
applied to solve this reduced problem.

Although our discussion above considered only two atomic sets, the analysis
extends trivially to any number of atomic sets.

  \begin{example}[Separating background from foreground in a noisy image] \label{example-mca}

    We give a concrete example from morphological component analysis that
    illustrates how this approach can be used in practice to separate background
    and foreground from a noisy image. Suppose that the $m$-vector $b=x_s +
    x_\ell + \epsilon$ encodes a 2-dimensional image comprised of a sparse
    component $x_s$, a low-rank component $x_\ell$, and structured noise
    $\epsilon$. The ability to decouple $b$ into these three components rests on
    their incoherence\cite{wright2009robust,oymak2017universality,McCoy2014}.
    Because our aim here is only to illustrate the polar-alignment property, we
    make the simplifying assumption that the noise $\epsilon$ is sparse in the
    Fourier basis, which is known to be incoherent with sparsity and low-rank.
    Based on these assumptions, we choose $\Ascr_1$ to be the unit 1-norm ball
    (\Cref{example-one-norm}), $\Ascr_2$ to be the nuclear-norm ball
    (\Cref{example-nuclear-norm}), and $\Ascr_3=D\T\Ascr_1$, where $D$ is the
    discrete cosine transform. Use \cref{prop-linear-transform} to deduce that
    the gauge that corresponds to $\Ascr_3$ is the transformed 1-norm:
    \[
      \gauge_{D\T\Ascr_1} = \gauge_{\Ascr_1}(D\,\cdot\,) = \|D\cdot\|_1.
    \]

    We follow the approach outlined in
    \cref{sec-morphological-component-analysis}. For the first stage, we apply
    the dual CG method to the problem
    \[
      \minimize{x}\enspace\half\|x-b\|_2^2\enspace\st\enspace\gauge_{\Ascr_1+\Ascr_2+\Ascr_3}(x)\le\tau
    \]
    to obtain the negative gradient $z^*=b-x^*$ (without storing the primal
    iterates $x\k$ or solution $x^*$). In the second stage, the primal solution
    $x^*$ is recovered by solving the problem
    \[
      \minimize{c^{(1)}_a,\,c^{(2)}_a,\,c^{(3)}_a}\enspace
      \half\|x-b\|_2^2 \text{with} x=\sum_{i=1,2,3}\sum_{a^{(i)}\in\Escr_{\Ascr_i}(z^*)}c_a^{(i)}a^{(i)}
    \]
    over the coefficients $c^{(i)}_a$. (Because in this case the atomic sets are
    centrosymmetric, we may ignore the nonnegativity requirements of the
    coefficients.)
    
    The first panel in \Cref{figure-chess-board} shows a noisy 500-by-500 pixel
    image of a chess board. The remaining panels show the separated images
    obtained after 2000 iterations of the CG algorithm as described above.

  \begin{figure}[t] \label{figure-chess-board}
    \centering
    \begin{tabular}{@{}c@{}c@{}c@{}c@{}}
             noisy       & low-rank   & sparse     & reconstructed
   \\[ -0pt] observation & background & foreground & image
   \\[-10pt]\includegraphics[width=.24\textwidth]{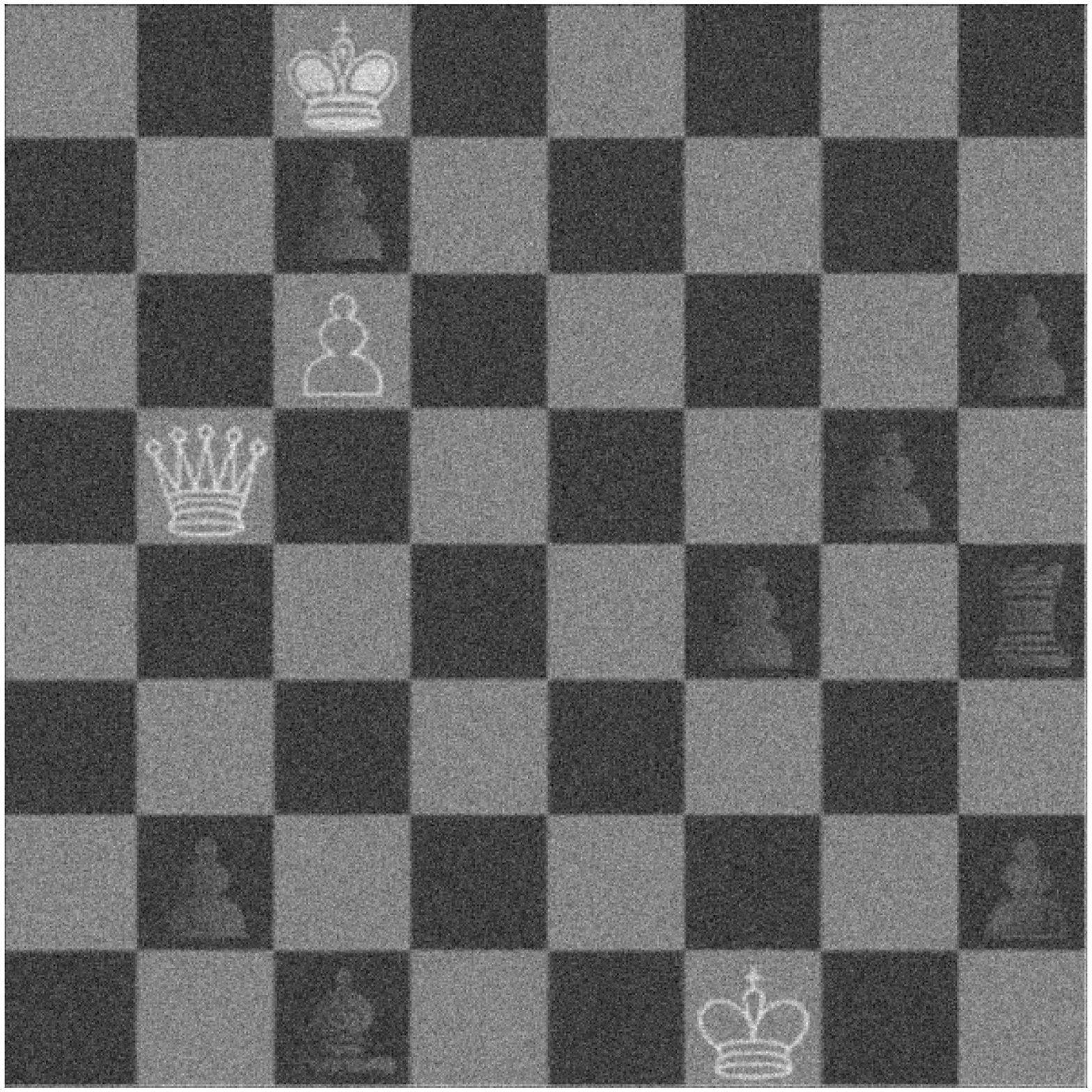}
    &\includegraphics[width=.24\textwidth]{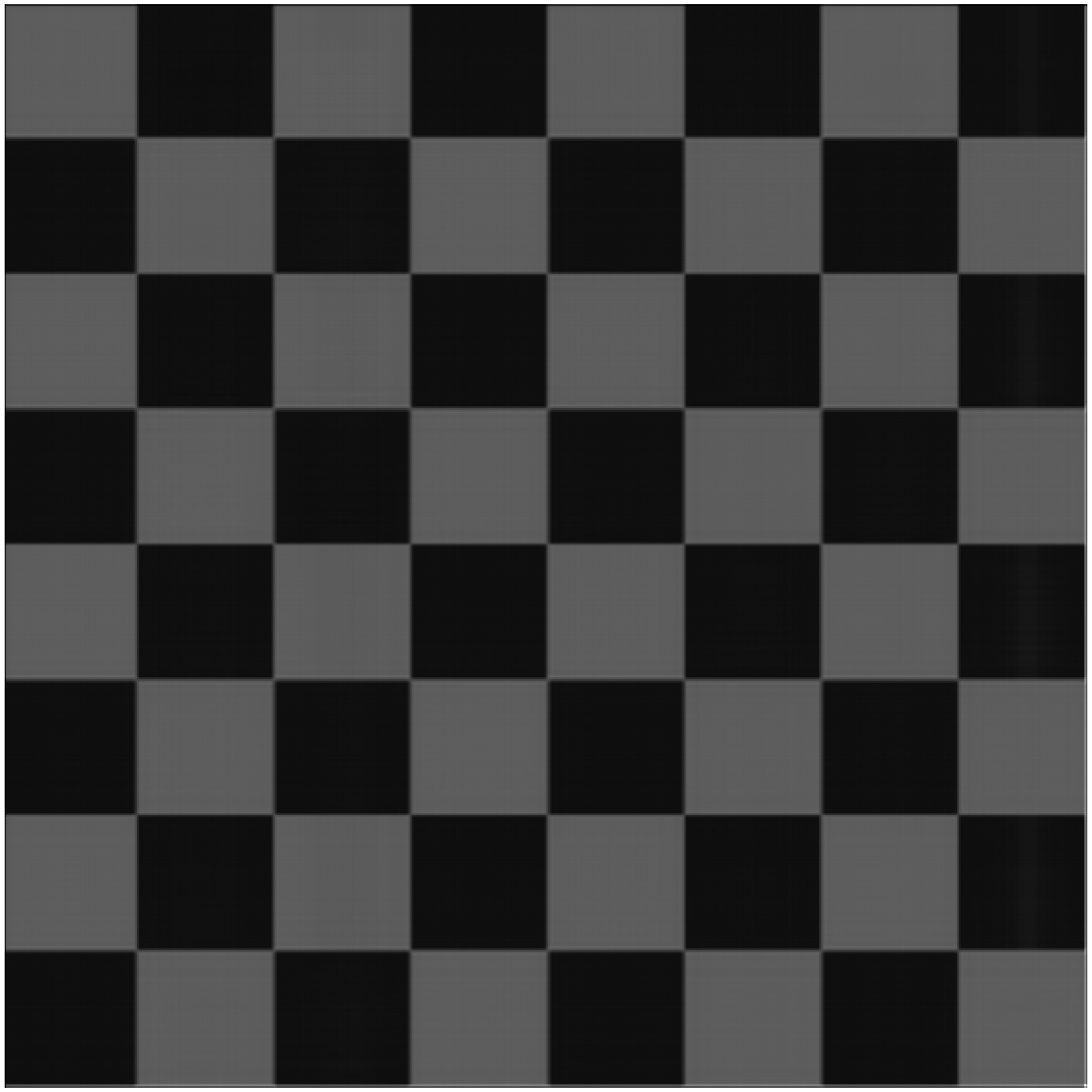}
    &\includegraphics[width=.24\textwidth]{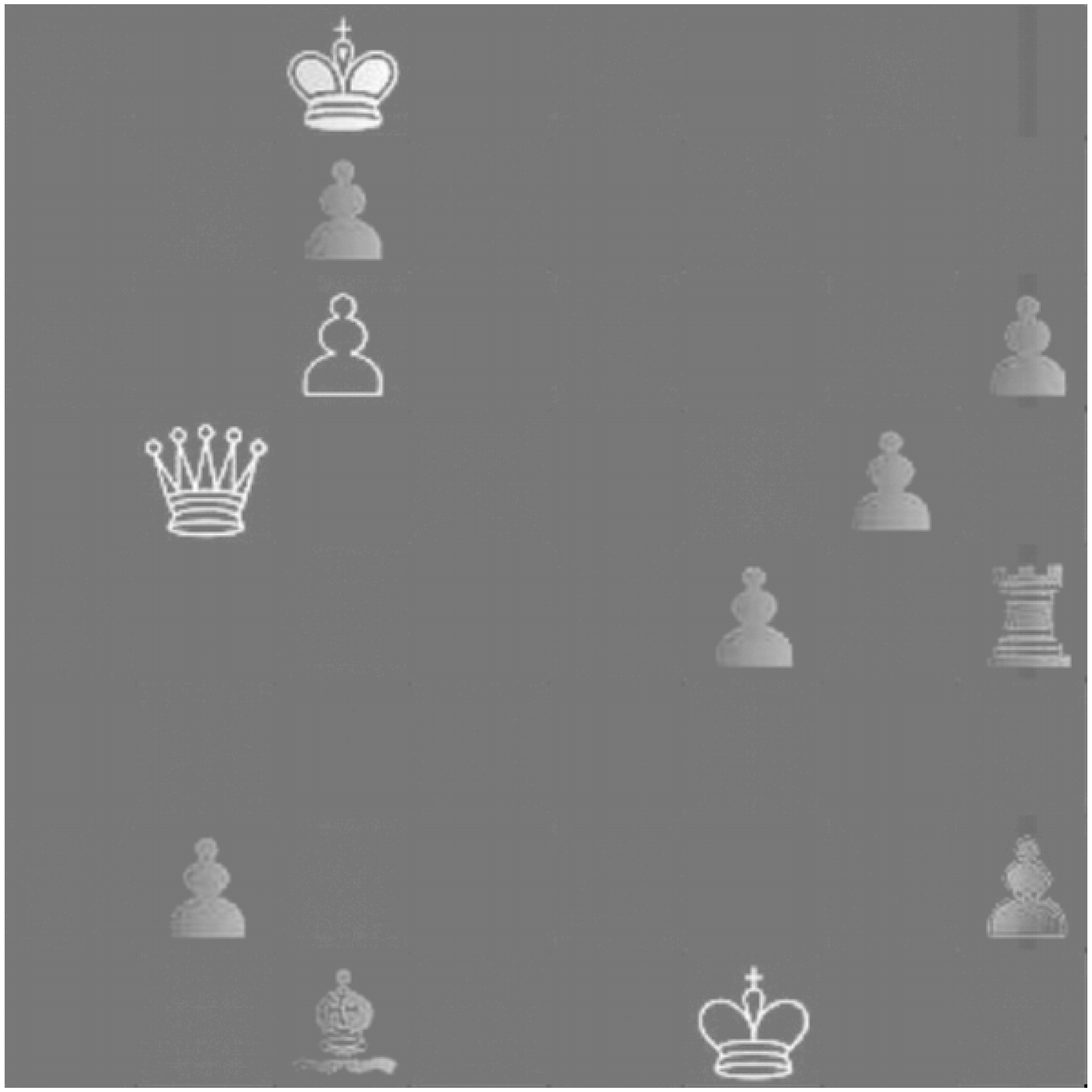}
    &\includegraphics[width=.24\textwidth]{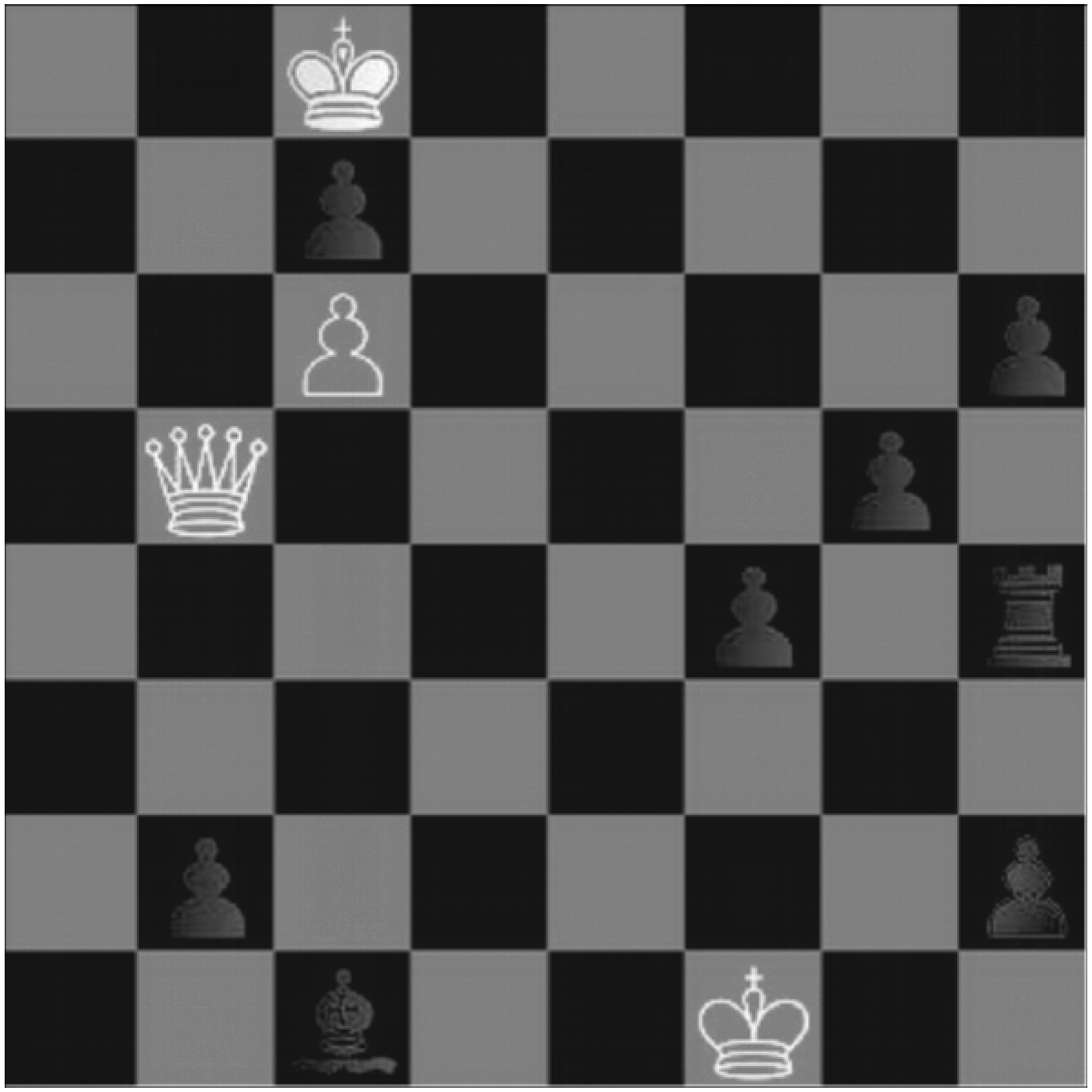}
  \\[-14pt] (a)         & (b)        & (c)        & (d) 
    \end{tabular}
    \caption{Morphological component analysis via polar convolution is used to denoise and separate
      foreground from background in an image. The chess-board image in panel (a)
    has been corrupted with noise. Panels (b) and (c) show the extracted
    low-rank and sparse parts of the image, which are assembled in panel (d) as
    the final reconstruction of the observed image (a). See \Cref{example-mca}.}
  \end{figure}
\end{example}

\subsection{Atomic unions and sum convolution}

The infimal sum convolution between two gauges $\gauge\Aso$ and
$\gauge\Ast$ is defined through the optimization problem
\begin{align*}
  (\gauge\Aso\infc\gauge\Ast)(x)
  &= \inf_w\set{\gauge\Aso(w)+\gauge\Ast(x-w)},
\end{align*}
Although here we define this operation only for gauges, it can be
applied to any two convex functions and always results in another
convex function~\cite[Theorem~5.4]{rockafellar1970convex}. Normally
the operation is simply called \emph{infimal convolution}, but here we
use the term \emph{sum convolution} to distinguish it from another
form of infimal convolution that we use in \cref{sec:sum-max-conv}.

\begin{proposition}[Sum convolution of gauges] \label{prop:sum-convolution}
  Let $\Ascr_1$ and $\Ascr_2$ be non-empty closed convex sets that
  contain the origin. The sum convolution of the gauges
  $\gauge\Aso$ and $\gauge\Ast$ is the gauge
  \[
    \gauge\Aso\infc\gauge\Ast =
    \gauge_{\scriptscriptstyle\Ascr_1\cup\Ascr_2}.
  \]
\end{proposition}
\begin{proof}
  Using \cref{prop-guage-equivalence}, we are led to the following equivalent expressions:
  \begin{align*}
    (\gauge\Aso\infc\gauge\Ast)(x)
    &= \inf_w\Biggl\{
      \inf_{c_{a}}
      \biggl\{
        \sum_{a\in\Ascr_1}c_a \biggm
        |\ \phantom{x-}w=\sum_{a\in\Ascr_1}c_aa,\ c_a\ge0
      \biggr\}
      \\&\phantom{\inf\Biggl\{\,}+
      \inf_{c_{a}}
      \biggl\{
        \sum_{a\in\Ascr_2}c_a \biggm|
        x-w=\sum_{a\in\Ascr_2}c_aa,\ c_a\ge0
      \biggr\}      
    \Biggr\}
    \\
    &= \inf_{w,\,c_a}\Biggl\{
      \sum_{a\in\Ascr_1\cup\Ascr_2}c_a \biggm|
      w = \sum_{a\in\Ascr_1}c_aa,\
      x-w = \sum_{a\in\Ascr_2}c_aa,\ c_a\ge0
      \biggr\}
  \\&= \inf_{c_a}\Biggl\{
      \sum_{a\in\Ascr_1\cup\Ascr_2}c_a \biggm|
      x = \sum_{\mathclap{a\in\Ascr_1\cup\Ascr_2}}c_aa,\ c_a\ge0
    \biggr\}
    = \gauge_{\scriptscriptstyle\Ascr_1\cup\Ascr_2}(x),
  \end{align*}
  which establishes the claim.
\end{proof}

\section{Conclusions} \label{sec-conclusions}

The theory of polar alignment and its relationship with atomic decompositions
offers a rich grammar with which to think about structured optimization. Of
course, the underlying ideas are not entirely new and many of the conclusions can be
derived using standard arguments from Lagrange multiplier theory, but we have
found that the notions of polarity and alignment offer a clarifying viewpoint.
Indeed, concepts such as active sets and supports, which are intuitive for
polyhedral constraints and vectors, easily extend to more abstract settings when
we adopt the vocabulary of alignment, exposed faces, and the machinery of 
gauges and support functions.

Further research opportunities remain. For example, most (if not all) of the
ideas we have presented could be generalized to the infinite-dimensional
setting, which would accommodate more general decompositions. Also, other
standard algorithms, such as splitting and bundle
methods~\cite{FanSunFriedlander2019}, seem to exhibit properties that can easily
be explained using the language of polar alignment.

\section*{Acknowledgments}

The last author (MPF) is indebted to James Burke at the University of
Washington, who over years of collaboration and discussion, and through his
mastery of convex analysis and optimization, helped to influence many of the
ideas in this manuscript.

\bibliographystyle{plain}
\bibliography{bibmaster/shorttitles,bibmaster/master,bibmaster/friedlander}
\end{document}